\documentclass[11pt, reqno]{amsart}
\usepackage[marginratio=1:1,height= 600pt, width=445pt,tmargin =90pt]
{geometry}
\usepackage{amsmath}
\usepackage{amsfonts}
\usepackage{amssymb}
\usepackage[all]{xy}           
\usepackage{bm}
\usepackage{bbding}
\usepackage{txfonts}
\usepackage{amscd}
\usepackage[english]{babel}
\usepackage{xspace}
\usepackage[shortlabels]{enumitem}
\usepackage{ifpdf}

\ifpdf
\usepackage[colorlinks,final,backref=page,hyperindex]{hyperref}
\else
\usepackage[colorlinks,final,backref=page,hyperindex,hypertex]{hyperref}
\fi
\usepackage{tikz}
\usepackage[active]{srcltx}
\usepackage{xcolor}
\definecolor{rose}{rgb}{1.5,0.0,0.60}

\makeatletter

\newtheorem{thm}{Theorem}[section]
\newtheorem{lem}[thm]{Lemma}

\newtheorem{pro}[thm]{Proposition}
\newtheorem{ex}[thm]{Example}
\theoremstyle{definition}

\newtheorem{defi}[thm]{Definition}

\newcommand{\nc}{\newcommand}
\newcommand{\delete}[1]{}

	\nc{\mlabel}[1]{\label{#1}}  
	\nc{\mcite}[1]{\cite{#1}}  
	\nc{\mref}[1]{\ref{#1}}  
	\nc{\mbibitem}[1]{\bibitem{#1}} 

\delete{
	\nc{\mlabel}[1]{\label{#1}{\hfill \hspace{1cm}{\bf{{\ }\hfill(#1)}}}}
	\nc{\mcite}[1]{\cite{#1}{{\em{{\ }(#1)}}}}  
	\nc{\mref}[1]{\ref{#1}{{\em{{\ }(#1)}}}}  
	\nc{\mbibitem}[1]{\bibitem[\em #1]{#1}} 
}

\newcommand {\emptycomment}[1]{}

\nc{\oprn}{\theta}
\nc{\Oprn}{\Theta}

\nc{\calo}{\mathcal{O}}
\nc{\oop}{$\mathcal{O}$-operator\xspace}
\nc{\oops}{$\mathcal{O}$-operators\xspace}
\nc{\mrho}{{\bm{\varrho}}}
\nc{\emk}{\mathbf{K}}
\nc{\invlim}{\displaystyle{\lim_{\longleftarrow}}\,}
\nc{\ot}{\otimes}

\newcommand{\be }{\begin{equation}}
	\newcommand{\ee }{\end{equation}}

\newcommand{\br}[1]{   [ \cdot,    \cdot  ]   }

\nc{\CV}{\mathbf{C}}

\begin{document}

	\title[Cohomologies of modified  Rota-Baxter Lie algebras with derivations and applications]
	{Cohomologies of modified Rota-Baxter Lie algebras with derivations and applications}

	\author{Basdouri Imed, Benabdelhafidh Sami, Sadraoui Mohamed Amin}
	\address{University of Gafsa, Faculty of Sciences Gafsa, 2112 Gafsa, Tunisia.}
	\email{\bf basdourimed@yahoo.fr}
	
	\address{University of Sfax, Faculty of Sciences of Sfax, BP 1171, 3038 Sfax, Tunisia.}
	\email{\bf abdelhafidhsami41@gmail.com}
	
	\address{University of Sfax, Faculty of Sciences of Sfax, BP 1171, 3038 Sfax, Tunisia.}
	\email{\bf aminsadrawi@gmail.com}
	
	\begin{abstract}
		In this paper, first, we introduce a notion of modified Rota-Baxter Lie algebras of weight $\mathrm{\lambda}$ with
		derivations (or simply modified Rota-Baxter LieDer pairs) and their representations. Moreover, we investigate cohomologies
		of a modified Rota-Baxter LieDer pairs with coefficients in a suitable representation. As applications,
		we study formal one-parameter deformations and abelian extensions of modified Rota-Baxter LieDer pairs.
	\end{abstract}
	
	
	\keywords{Lie algebras, derivation, modified Rota-Baxter operator, cohomology, deformation, extension.}
	
	\maketitle
	
	\vspace{-1.1cm}
	
	\tableofcontents
	
	\allowdisplaybreaks
	
	\section{Introduction}

	In 1960, Baxter introduced the notion of Rota-Baxter operators on associative algebras in his study of fluctuation theory in probability \cite{G0}. Rota-Baxter operators have found many applications in mathematics and physics, such as renormalizations in perturbative quantum field theory \cite{A0}, combinatorics \cite{N0}, Yang Baxter equations \cite{C0}, multiple zeta values in number theory \cite{L0}, and algebraic operad \cite{M0}. Das, Hazra, and Mishra considered Rota-Baxter Lie algebras in \cite{S2}. Cohomologies, deformations, and extensions of Rota-Baxter Leibniz algebras were established in \cite{B2}. Jiang and Sheng constructed cohomologies of relative Rota-Baxter Lie algebras with coefficients in an arbitrary representation \cite{J1}. The Rota-Baxter operator of arbitrary weights on Lie algebras was studied in \cite{A3,K0}. Cohomologies theory of Rota-Baxter pre-Lie algebras of arbitrary weights was studied in \cite{S1}. Rota-Baxter Lie triple systems of any weights were established in \cite{S0}.
	
	In \cite{M2}, Semenov-Tian-Shansky solved the solution of the modified classical Yang-Baxter equation, which was called the modified $r$-matrix in \cite{Jun}. Inspired by the case of the modified $r$-matrix, this modified algebraic structure has been extended to other algebraic structures, such as modified Rota-Baxter associative algebras of weight $\lambda$ \cite{A1}, modified $\lambda$-differential Lie algebras in \cite{X0}, modified Rota-Baxter Leibniz algebras of weight $\lambda$ \cite{Y0,B1}, modified $\lambda$-differential Lie triple systems \cite{W1}, and modified Rota-Baxter Lie-Yamaguti algebras in \cite{W0}.
	
	Derivations are also useful in constructing homotopy Lie algebras \cite{T0}, deformation formulas \cite{V2}, and differential Galois theory \cite{A5}. They also play an important role in control theory and gauge theories in quantum field theory \cite{V0,V1}. In \cite{M1,J2}, the authors considered algebras with derivations from the operadic point of view. Recently, Lie algebras with a derivation (called LieDer pairs) are studied from a cohomological point of view \cite{R0}, and deformations, extensions of LieDer pairs are studied. The results of \cite{R0} have been extended to associative algebras with derivations (called AssDer pairs) in \cite{A2}, Leibniz algebras with derivations are established in \cite{A4}, and cohomologies and deformations of Lie triple systems with derivations are studied in \cite{Q0}. Recently, cohomologies of relative Rota-Baxter Lie algebras with derivations and applications are considered in \cite{Q1}. Also derivations play an important role in the construction of the InvDer algebraic structures in \cite{Bas}. 
	
	Motivated by these works, we are devoted to developing cohomologies of modified Rota-Baxter LieDer pairs and applying them to the formal deformation and abelian extension of modified Rota-Baxter LieDer pairs.
	
	The paper is organized as follows. In section \ref{sec2}, we consider modified Rota-Baxter LieDer pairs and introduce their representations. In section \ref{sec3}, we define the cohomology of modified Rota-Baxter LieDer pairs with coefficients in a suitable representation. In section \ref{sec4}, we study formal deformation theory and rigidity of modified Rota-Baxter LieDer pairs. Finally, in section \ref{sec5}, we discuss an abelian extension of the modified Rota-Baxter LieDer pairs and characterize extensions in terms of our second cohomology groups.
	
	Throughout this paper, let $\mathbb{K}$ be a field of characteristic $0$. Except specially stated, vector spaces are $\mathbb{K}$-vector spaces, and all tensor products are taken over $\mathbb{K}$.
	
	
	\section{Modified Rota-Baxter LieDer pair}\label{sec2}
	\def\theequation{\arabic{section}.\arabic{equation}}
	\setcounter{equation} {0}
	In this section, we consider modified Rota-Baxter LieDer pairs and introduce their representations. We also provide various examples and new construction.
	\\
	
	A {\bf modified Rota-Baxter Lie algebra} of weight $\lambda$ consists of a {\bf Lie algebra} $\mathcal{A}=\mathrm{(A,[-,-])}$ equipped
	with a {\bf modified Rota-Baxter operator} of weight $\lambda$ denoted by $\mathrm{R}$ such that
	\begin{equation}\label{modified RBO}
		\mathrm{[Ra,Rb]=R([Ra,b]+[a,Rb])+\lambda [a,b],\quad \forall a,b\in A} .
	\end{equation}
	
	Inspired by the definition of modified Rota-Baxter Lie algebra and the notion of {\bf LieDer pair} \cite{R0} we
	introduce the following.
	\begin{defi}
		A \textbf{modified Rota-Baxter LieDer pair of weight $\lambda$} consists of a modified Rota-Baxter Lie algebra  $(\mathcal{A},\mathrm{R}
		)$ equipped with a
		\textbf{derivation} $\mathrm{d:A\rightarrow A}$  such that
		\begin{equation}\label{condition1 MRBLieDer pair}
			\mathrm{R\circ d=d\circ R}.
		\end{equation}
		Denote it by $(\mathcal{A},\mathrm{R,d})$.
	\end{defi}
	\begin{ex}
		Let $\{\mathrm{e_1,e_2}\}$ be a basis of a $2$-dimensional vector space $\mathrm{A}$ over $\mathbb{R}$. Given a Lie structure
		$\mathrm{[e_1,e_2]=e_2}$, then the triple $(\mathcal{A},\mathrm{R,d})$ is a modified Rota-Baxter LieDer pair of weight $(-\lambda)$ with
		\begin{align*}
			\mathrm{d=\begin{pmatrix}
					0 & 0 \\
					0 & a_{22}
				\end{pmatrix}\quad \text{ and   } \quad	R=\begin{pmatrix}
					a_{11} & 0 \\
					0 & \sqrt{\lambda}
			\end{pmatrix}},\quad \text{ where } (\lambda>0).
		\end{align*}
	\end{ex}
	Note that a modified Rota-Baxter LieDer pair $(\mathcal{A},\mathrm{R,d})$ of weight $0$
	is just a Rota-Baxter LieDer pair.
	In the sequel we denote modified Rota-Baxter LieDer pair instead of modified Rota-Baxter LieDer pair of weight $\lambda$ if there is no confusion.
	\begin{defi}
		A \textbf{morphism} of modified Rota-Baxter LieDer pairs from $(\mathcal{A}_1,\mathrm{R_1,d_1})$ to $(\mathcal{A}_2,\mathrm{R_2,d_2})$ is a Lie
		algebra morphism $\varphi:\mathcal{A}_1\rightarrow \mathcal{A}_2$ such that the following identities holds
		\begin{eqnarray}
			\varphi\circ \mathrm{d_1}&=&\mathrm{d_2}\circ\varphi,\label{morphism RBO2}\\
			\varphi\circ \mathrm{R_1}&=&\mathrm{R_2}\circ\varphi.\label{morphism RBO3}
		\end{eqnarray}
	\end{defi}

	\cite{A3} Recall that a linear map $\mathrm{T : A\rightarrow A}$ is called a Rota-Baxter operator of weight $\lambda$ on the
	Lie algebra $\mathcal{A}$ if $\mathrm{T}$ satisfies
	\begin{align*}
		\mathrm{[T(a),T(b)]}&=\mathrm{ T([T(a),b] + [a,T(b)] + \lambda[a,b]}),\quad \forall \mathrm{a,b \in A}.
	\end{align*}
	\cite{B0} A \textbf{Rota-Baxter LieDer pair} of weight $\lambda$ is a triple $(\mathcal{A},\mathrm{T,d})$
	consisting of a Lie algebra $\mathcal{A}$ together with a derivation $\mathrm{d}$ and a Rota-Baxter operator $\mathrm{T}$ of
	weight $\lambda$ on it such that
	\begin{equation}\label{RB LieDer pair1}
		\mathrm{T \circ d=d\circ T}
	\end{equation}
	Furthermore, we have
	
	\begin{lem}\label{lem1}
		The triple $(\mathcal{A},\mathrm{T,d})$ is a Rota-Baxter LieDer pair if and only if $(\mathcal{A},\mathrm{2T+\lambda Id_A,d})$ is a modified
		Rota-Baxter LieDer pair of weight $-\lambda^2$.
	\end{lem}
	\begin{proof}
		For $\mathrm{a,b\in A}$, we have
		\begin{align*}
			[(&\mathrm{2T+\lambda Id_A)(a),(2T+\lambda Id_A)(b)}]\\
			=& \mathrm{[2T(a)+ \lambda a, 2T(b)+ \lambda b]} \\
			=& \mathrm{4[T(a),T(b)] + 2\lambda[T(a),b] + 2\lambda[a,T(b)] + \lambda^2[a,b]}\\
			=& \mathrm{4T\big([T(a),b] + [a,T(b)] + \lambda[a,b]\big)+2\lambda[T(a),b] + 2\lambda[a,T(b)] + \lambda^2[a,b]} \\
			=& \mathrm{(2T + \lambda Id_A)\big([(2T + \lambda Id_A)(a),b] + [a, (2T +\lambda Id_A)(b)]\big) - \lambda^2[a,b],}
		\end{align*}
		and
		\begin{align*}
			\mathrm{(2T+\lambda Id_A)\circ d(a)}&=\mathrm{2T\circ d(a)+\lambda d(a)}\\
			&\overset{\eqref{RB LieDer pair1}}{=}\mathrm{2d\circ T(a)+\lambda Id_A \circ d(a)}\\
			&=\mathrm{d\circ (2T+\lambda Id_A)(a)}.
		\end{align*}
		Thus, $(\mathcal{A},\mathrm{ 2T + \lambda Id_A,d})$ is a modified Rota-Baxter LieDer pair of
		weight $-\lambda^2$.
	\end{proof}
	Let $(\mathcal{A},\mathrm{d})$ be a LieDer pair. Recall that, from \cite{R0}, a representation of it is a vector
	space $\mathrm{V}$
	with two linear maps $\rho:\mathrm{A}\rightarrow \mathrm{gl(V)}$ and
	$\mathrm{d_\mathrm{V}:V\rightarrow V}$ such that, for all $\mathrm{a,b\in A}$
	\begin{eqnarray*}
		\mathrm{\rho([a,b])}&=&\mathrm{\rho(a)\circ \rho(b)-\rho(b)\circ \rho(a)},\\
		\mathrm{d_\mathrm{V}\circ \rho(a)}&=&\mathrm{\rho(d a)+\rho(a)\circ d_\mathrm{V}}.
	\end{eqnarray*}
	\begin{defi}
		Let $(\mathcal{A},\mathrm{R,d})$ be a modified Rota-Baxter LieDer pair. A \textbf{representation} of it is a triple
		$(\mathcal{V}=\mathrm{(V;\rho),R_V,d_\mathrm{V}})$ where $\mathcal{V}$ is a representation of the Lie algebra $\mathcal{A}$,
		$\mathrm{R_V:V\rightarrow V}$ and $\mathrm{d_V:V\rightarrow V}$ are a linear maps such that for all
		$\mathrm{a,b\in A}$ and $\mathrm{u\in V}$
		\begin{eqnarray}
			\mathrm{\rho(Ra)(R_Vu)}&=&\mathrm{R_V(\rho(Ra)(u)+\rho(a)(R_Vu))+\lambda \rho(a)u},\label{Rep of RBDer pair2}\\
			\mathrm{d_V (\rho(a)(u))}&=&\mathrm{\rho(d a)(u)+\rho(a) (d_V(u))},\label{Rep of RBDer pair1}\\
			\mathrm{R_V( d_V(u))}&=&\mathrm{d_V (R_V(u))}. \label{Rep of RBDer pair3}
		\end{eqnarray}
	\end{defi}
	\begin{ex}
		Any modified Rota-Baxter LieDer pair $(\mathcal{A},\mathrm{R,d})$ is a representation of itself. Such a representation is
		called \textbf{the adjoint representation}.
	\end{ex}
	\begin{ex}
		Let $(\mathcal{A},\mathrm{R,d})$ be a modified Rota-Baxter LieDer pair and $(\mathcal{V},\mathrm{R_V,d}_\mathrm{V})$ be a
		representation of it. Then for any scalar $\kappa\in \mathbb{K}$, the triple $\mathrm{(\mathcal{V},\kappa R_V,d_V)}$ is a representation of the $(\kappa\lambda)$-modified Rota-Baxter LieDer pair $(\mathcal{A},\mathrm{\kappa R,d_\mathrm{V}})$.
	\end{ex}
	\begin{ex}
		Let $(\mathrm{\mathcal{A},R,d})$ be a modified Rota-Baxter LieDer pair and $\mathrm{(\mathcal{V},R_V,d_V)}$ be a
		representation of it. Then the triple
		$\mathrm{(\mathcal{V},-\lambda \mathrm{Id}_\mathrm{V}-R_V,d_\mathrm{V})}$ is a representation of the
		modified Rota-Baxter LieDer pair $\mathrm{(\mathcal{A},-\lambda\mathrm{Id}_\mathrm{A}-R,d)}$.
	\end{ex}

	Next, inspired by the reference \cite{R0} we construct the  \textbf{semi-direct product} in the context of modified Rota-Baxter
	LieDer pair.
	\begin{pro}
		Let $(\mathcal{A},\mathrm{R,d})$ be a modified Rota-Baxter LieDer pair and  $(\mathcal{V},\mathrm{R_V,d}_\mathrm{V})$ be a
		representation of it. Then $(\mathcal{A}\oplus \mathcal{V},\mathrm{R\oplus R_V,d\oplus d}_{\mathrm{V}})$ is a modified Rota-Baxter LieDer pair
		where the Lie bracket on $\mathcal{A}\oplus \mathcal{V}$ is given by
		\begin{equation*}
			\mathrm{[a+u,b+v]_\ltimes:=[a,b]+\rho(a)v-\rho(b)u,\quad \forall a,b\in A \quad \forall u,v\in V},
		\end{equation*}
		and the modified Rota-Baxter operator is given by
		\begin{eqnarray*}
			\mathrm{(R\oplus R_V)(a+u)=R a+R_Vu},
		\end{eqnarray*}
		and the derivation is given by
		\begin{eqnarray*}
			\mathrm{(	d\oplus d_\mathrm{V})(a+u)=d a+d_\mathrm{V}u}
		\end{eqnarray*}
		
		We call such structure by the semi-direct product of the modified Rota-Baxter LieDer pair $(\mathcal{A},\mathrm{R,d})$
		by a representation of it $(\mathcal{V},\mathrm{R_V,d}_\mathrm{V})$ and denoted by $\mathcal{A}\ltimes_{\mathrm{mRBLD^\lambda}}\mathcal{V}$.
	\end{pro}
	\begin{proof}
		For any $\mathrm{a,b\in A}$ and $\mathrm{u,v\in V}$, we have
		\begin{align*}
			\mathrm{[R\oplus R_V(a+u),R\oplus R_V(b+v)]_{\ltimes}}
			=&\mathrm{[Ra,Rb] + \rho(Ra)R_Vv-\rho(Rb)R_Vu}\\
			=&\mathrm{R\big([Ra,b] + [a,Rb]\big)+\lambda[a,b]} \\
			&+\mathrm{ R_V\big(\rho(Ra)v+\rho(a)(R_Vv)\big) +\lambda\rho(a)v} \\
			&- \mathrm{R_V\big(\rho(Rb)u+\rho(b)(R_Vu)\big)-\lambda\rho(b)u}\\
			=&\mathrm{R \oplus R_V \big([R \oplus R_V(a+u),b+v]_{\ltimes}}\\
			&+\mathrm{[a+u, R \oplus R_V(b+v)]_{\ltimes}\big)+\lambda[a+u,b+v]_{\ltimes}},
		\end{align*}
		similarly, we have
		\begin{align*}
			\mathrm{[d\oplus d_V(a+u),d\oplus d_V(b+v)]_{\ltimes}}=&\mathrm{[d\oplus d_V(a+u),b+v]_{\ltimes}} \\
			&+\mathrm{[a+u,d\oplus d_V(b+v)]_{\ltimes}},
		\end{align*}
		and
		\begin{align*}
			\mathrm{(R\oplus R_V)\circ(d\oplus d_V)(a+u)}&=\mathrm{(R\oplus R_V)(da+d_\mathrm{V}u)}\\
			&=\mathrm{R(da)+R_\mathrm{V}(d_\mathrm{V}u)}\\
			&=\mathrm{d(Ra)+d_\mathrm{V}(R_\mathrm{V}u)}\\
			&=\mathrm{(d\oplus d_V)\circ(R\oplus R_V)(a+u)}.
		\end{align*}
		This complete the proof.
	\end{proof}
	Next, we study \textbf{induced modified Rota-Baxter LieDer pair} from a modified Rota-Baxter operator of weight $\lambda$.
	\begin{pro}
		Let $(\mathcal{A},\mathrm{R,d})$ be a modified Rota-Baxter LieDer pair. Define the following bracket
		\begin{equation}\label{induced RBLieDer}
			\mathrm{[a,b]_\mathrm{R}:=[Ra,b]+[a,Rb]},\quad \forall \mathrm{a,b\in A}.
		\end{equation}
		Then we have
		\begin{enumerate}
			\item[1)] $\mathrm{(A,[-,-]_R)}$ is a new Lie algebra, we denote it by $\mathcal{A}_R$,
			\item [2)] the couple $(\mathcal{A}_R,\mathrm{R})$ is a modified Rota-Baxter Lie algebra of weight
			$\mathrm{\lambda}$,
			\item [3)] the triple $\mathrm{(\mathcal{A}_\mathrm{R},R,d)}$ is a modified Rota-Baxter LieDer pair.
		\end{enumerate}
	\end{pro}
	\begin{proof}
		For all $\mathrm{a,b\in A}$ and $\mathrm{u \in V}$, we have
		\begin{enumerate}
			\item [1)] it is easy to verify that $\mathrm{(A,[-,-]_R)}$ is a Lie algebra.
			\item [2)] We need to show that $\mathrm{R}$ is a modified Rota-Baxter operator on $\mathcal{A}$.
			\begin{align*}
				\mathrm{[Ra,Rb]_\mathrm{R}}
				=&\mathrm{[R^2a,Rb]+[Ra,R^2b]}\\
				=&\mathrm{R([R^2a,y]+[Ra,Rb])+\lambda[Ra,b]}\\
				&+\mathrm{R([Ra,Rb]+[a,R^2b])+\lambda[a,Rb]}\\
				=&\mathrm{R([Ra,b]_\mathrm{R}+[a,Rb]_\mathrm{R})+\lambda[a,b]_\mathrm{R}}.
			\end{align*}
			\item [3)] Now, we need to prove that $\mathrm{d}$ is a derivation on $\mathrm{(A_R,R)}$,
			\begin{align*}
				\mathrm{d([a,b]_R)}
				=&\mathrm{d([Ra,b]+[a,Rb])}\\
				=&\mathrm{d([Ra,b])+d ([a,Rb])}\\
				=&\mathrm{[d(Ra),b]+[Ra,db]+[da,Rb]+[a,d(Rb)]}\\
				\overset{\ref{condition1 MRBLieDer pair}}{=}&\mathrm{[R(d a),b]+[Ra,d b]+[d a,Rb]+[a,R (db)]}\\
				=&\mathrm{\Big([R(d a),b]+ [d a,Rb]\Big)+\Big([Ra,d b]+[a,R( db)]\Big)}\\
				=&\mathrm{[d a,b]_R+[a,d b]_R},
			\end{align*}
			Thus, by equation \eqref{condition1 MRBLieDer pair} we obtain the result.
		\end{enumerate}
	\end{proof}
\begin{thm}\label{theorem needed in the cohomology}
	Let $(\mathcal{A},\mathrm{R,d})$ be a modified Rota-Baxter LieDer pair and $(\mathrm{V,R_V,d_{\mathrm{V}})}$ be a representation of it. Define a map
	\begin{equation}\label{rep of new Rota-Baxter LieDer pair}
		\mathrm{\rho_R(a)(u)=\rho(Ra)(u)-R_V(\rho(a)(u)),\quad \forall a\in A , u\in V}.
	\end{equation}
	Then $\mathrm{\rho_R}$ defines a representation of the LieDer pair $(\mathcal{A}_{\mathrm{R}},\mathrm{d})$ on $(\widetilde{\mathcal{V}},d_{\mathrm{V}})=(\mathrm{(V;\rho_R),d}_{\mathrm{V}})$.
	Moreover, $(\widetilde{\mathcal{V}},\mathrm{R_V,d_{V}})$ is a representation of the modified Rota-Baxter LieDer pair $(\mathcal{A}_{\mathrm{R}},\mathrm{R,d})$.
\end{thm}
\begin{proof}
	Let $\mathrm{a,b\in A}$ and $\mathrm{u\in V}$.
	\begin{align*}
		&\mathrm{\rho_R(a)\circ \rho_R(b)u-\rho_R(b)\circ \rho_R(a)u}\\
		=&\mathrm{\rho_R(a)(\rho(Rb)u-R_V\rho(b)u)-\rho_R(b)(\rho(Ra)u-R_V\rho(a)u)}\\
		=&\mathrm{\rho(Ra)\rho(Rb)u-\rho(Ra)R_V\rho(b)u-R_V\big(\rho(a)\rho(Rb)u-\rho(a)R_V\rho(b)u\big)}\\
		&-\mathrm{\rho(Rb)\rho(Ra)u+\rho(Rb)R_V\rho(a)u+R_V\big(\rho(b)\rho(Ra)u-\rho(b)R_V\rho(a)u\big)}\\
		=&\mathrm{\rho(Ra)\rho(Rb)u-R_V\big(\rho(Ra)\rho(b)u+\rho(a)R_V\rho(b)u\big)-\lambda\rho(a)\rho(b)u}\\
		&-\mathrm{R_V\big(\rho(a)\rho(Rb)u-\rho(a)R_V\rho(b)u\big)-\rho(Rb)\rho(Ra)u}\\
		&+\mathrm{R_V\big(\rho(Rb)\rho(a)u+\rho(b)R_V\rho(a)u\big)+\lambda\rho(b)\rho(a)u+R_V\big(\rho(b)\rho(Ra)u-\rho(b)R_V\rho(a)u\big)}\\
		=&\mathrm{\rho([Ra,Rb])u-R_V\big(\rho([Ra,b]+[a,Rb])\big)u-\lambda\rho([a,b])u}\\
		=&\mathrm{\rho(R[a,b]_\mathrm{R})+\lambda\rho([a,b])-R_V([a,b]_\mathrm{R})-\lambda\rho([a,b])}\\
		=&\mathrm{\rho(R[a,b]_\mathrm{R})-R_V([a,b]_\mathrm{R})}\\
		=&\mathrm{\rho_R([a,b]_\mathrm{R})u}.
	\end{align*}
	Wich means that $\mathrm{\rho_R}$ is a representation of $\mathcal{A}_\mathrm{R}$ on $\mathrm{V}$ in the context of Lie algebra structure. So we need just to show that $(\mathrm{\rho_R,d_V})$ is a representation of $(\mathrm{\mathcal{A}_R,d})$. Let $\mathrm{a\in A}$ and $\mathrm{u\in V}$, by using equations \eqref{condition1 MRBLieDer pair},\eqref{Rep of RBDer pair1} and \eqref{Rep of RBDer pair3} we have
	\begin{align*}
		\mathrm{d_{V}\circ \rho_R(a)u}
		=&\mathrm{d_{V}\circ (\rho(Ra)u-R_V(\rho(a)u))}\\
		=&\mathrm{d_{V}\circ \rho(Ra)u-d_{V}(R_V(\rho(a)u))}\\
		\overset{}{=}&\mathrm{\rho(d\circ Ra)u+\rho(Ra)\circ d_{V}u-d_{V}(R_V(\rho(a)u))}\\
		\overset{}{=}&\mathrm{\rho(d\circ Ra)u+\rho(Ra)\circ d_{V}u-R_V(d_{V}(\rho(a)u))}\\
		\overset{}{=}&\mathrm{\rho(d\circ Ra)u+\rho(Ra)\circ d_{V}u-R_V(\rho(d a)u+\rho(a)\circ d_{V}u)}\\
		\overset{}{=}&\mathrm{\rho(R\circ d a)u-R_V(\rho(d a)u)+\rho(Ra)\circ d_{V}u-R_V(\rho(a)\circ d_{V}u)}\\
		=&\mathrm{\rho_R(d a)u+\rho_R(a)\circ d_{V}u}.
	\end{align*}
	For the next result we have
	\begin{align*}
		&\mathrm{R_V\big(\rho_R(Ra)u+\rho_R(a)(R_Vu) \big)+\lambda\rho_R(a)u}\\
		=&\mathrm{R_V\Big(\rho(R(Ra))u-R_V(\rho(Ra))u+\rho(Ra)(R_Vu)-R_V(\rho(a)(R_Vu))\Big)+\lambda\big(\rho(Ra)u-R_V(\rho(a)u)\big)}\\
		=&\mathrm{R_V\Big(\rho(R(Ra))u+\rho(Ra)(R_Vu)\Big)+\lambda\rho(Ra)u-R_V\Big(R_V(\rho(a)(R_Vu)+\rho(Ra)u)+\lambda\rho(a)u\Big)}\\
		=&\mathrm{\rho(R(Rx))(R_Vu)-R_V(\rho(Rx)(R_Vu))}\\
		=&\mathrm{\rho_R(Ra)(R_Vu)}.
	\end{align*}
	And since equation \eqref{Rep of RBDer pair3} holds we complete the proof.
\end{proof}
	
	\section{Cohomology of modified Rota-Baxter LieDer pair}\label{sec3}
	\def\theequation{\arabic{section}.\arabic{equation}}
	In this section, we define the cohomology of modified Rota-Baxter LieDer pair.\\
	We first recall the Chevally-Eilenberg cohomology of Lie algebra and the cohomology of  modified Rota-Baxter Lie
	algebra with coefficients in an arbitrary representation. Then we define the
	cohomology of modified Rota-Baxter LieDer pair.\\
	Let $\mathrm{\mathcal{A}=(A,[-,-])}$ be a Lie algebra, the Chevally-Eilenberg cohomology of $\mathcal{A}$ with
	coefficients in the representation $\mathcal{V}$ is given by the cohomology of the cochain complex
	$\mathrm{\Big(\mathrm{C}^{\star}(A;V),\delta_\mathrm{CE}^\star\Big)}$ where
	$\mathrm{\mathrm{C}^n(A;V)=\mathrm{Hom}(\wedge^nA,V)}$ for $\mathrm{n\geq0}$ and the coboundary map
	\begin{center}
		$\mathrm{\delta_\mathrm{CE}^n:\mathrm{C}^n(A;V)\rightarrow \mathrm{C}^{n+1}(A;V)}$
	\end{center}
	is given by
	\begin{align*}
		\mathrm{\delta_\mathrm{CE}^n(f_n)(\mathrm{a_1,\cdots,a_{n+1}})}
		=&\mathrm{\displaystyle\sum_{i=1}^{n+1}(-1)^{i+n}\mathrm{\rho(a_i)}\mathrm{f_n(a_1,\cdots,\hat{a}_i,\cdots,a_{n+1}})}\\
		&\mathrm{+\displaystyle\sum_{1\leq i<j\leq n+1}(-1)^{i+j+n+1}f_n(\mathrm{[a_i,a_j],a_1,\cdots,\hat{a}_i,\cdots,\hat{a}_j,\cdots,a_{n+1}})}.
	\end{align*}
	For $\mathrm{f_n\in C^n({A;V})}$ and $\mathrm{\mathrm{a_1,\ldots,a_{n+1}}\in A}$.
	\subsection{Cohomology of modified Rota-Baxter Lie algebra}
	Now, using the Chevalley-Eillenberg cohomology for the induced Lie algebra 
	$\mathrm{\mathcal{A}_\mathrm{R}=(A,[-,-]_\mathrm{R})}$ 
	with coefficients in the representation 
	$\mathrm{(\widetilde{\mathcal{V}}=(V;\rho_\mathrm{R}))}$, we define the following:\\
	For $\mathrm{n\geq 0}$, we have $\mathrm{C^n_\mathrm{mRBO^\lambda}(A;V)=\mathrm{Hom}(\wedge^nA,V)}$ and the 
	coboundary operator
	\begin{equation*}
		\mathrm{\delta_\mathrm{mRBO^\lambda}^n:C^n_\mathrm{mRBO^\lambda}(A;V)\rightarrow C^{n+1}_\mathrm{mRBO^\lambda}(A;V)}
	\end{equation*}
	is given as follows
	\begin{align*}
		\mathrm{\delta_\mathrm{mRBO^\lambda}^n(f_n)(\mathrm{a_1,\cdots,a_{n+1}})}
		=&\mathrm{\displaystyle\sum_{i=1}^{n+1}(-1)^{i+n}\rho_\mathrm{R}\mathrm{(a_i)}f_n
			(\mathrm{a_1,\cdots,\hat{a}_i,\cdots,a_{n+1}})}\\
		&\mathrm{+\displaystyle\sum_{1\leq i<j\leq n+1}(-1)^{i+j+n+1}f_n(\mathrm{[a_i,a_j]_\mathrm{R},a_1,\cdots,
				\hat{a}_i,\cdots,\hat{a}_j,\cdots,a_{n+1}})}\\
		&\mathrm{=\sum_{i=1}^{n+1}(-1)^{i+n}\rho(\mathrm{R(a_{i})})(f_n\mathrm{(a_1,\cdots,\hat{a}_i,\cdots,a_{n+1})})}\\
		&\mathrm{-\sum_{i=1}^{n+1}(-1)^{i+n}R_\mathrm{V}(\rho(\mathrm{a_{i}})f_n(\mathrm{a_1,\cdots,\hat{a}_i,\cdots,a_{n+1})})}\\
		&\mathrm{+\sum_{1\leq i<j\leq n+1}(-1)^{i+j+n+1}f_n\mathrm{([R(a_i),a_j]+[a_i,R(a_j)],a_1,\cdots,\hat{a}_i,\cdots,\hat{a}_j,\cdots,a_{n+1}})}.
	\end{align*}
	Then $\mathrm{\Big(\mathrm{C}_\mathrm{mRBO^\lambda}^{\star}(A;V),\delta_\mathrm{mRBO^\lambda}^\star\Big)}$ is a cochain complex.\\
	Now, motivated by the Proposition 3.2 of \cite{A3} and definition 4.1 of \cite{B1}, we introduce the following
	\begin{defi}
		Let $\mathrm{(\mathcal{A},R)}$ be a modified Rota-Baxter Lie algebra of weight $\mathrm{\lambda}$ and 
		$\mathrm{(\mathcal{V},R_\mathrm{V})}$ be a representation of it. We define a map
		\[\mathrm{\phi^{n} : C^{n}(A,V) \rightarrow C^{n}_\mathrm{mRBO^\lambda}(A,V)}\]
		as follows:
		\begin{align*}
			&\mathrm{\phi^0(f_0)=\mathrm{Id_V}};\\
			&\mathrm{\phi^n(f_n)(\mathrm{a_1,a_2,\ldots,a_n})=f_n(\mathrm{Ra_1,Ra_2,\ldots, Ra_n})}\\
			&\mathrm{-\sum _{1\leq i_1<i_2< \cdots < i_r\leq n ,r~ \mbox{odd}}{(-\lambda)^{\frac{r-1}{2}}} 
				(\mathrm{R_V}\circ f_n)(\mathrm{R(a_1),\ldots,a_{i_1},\ldots,a_{i_r},\ldots,R(a_n}))}\\
			&\mathrm{-\sum _{1\leq i_1<i_2< \cdots < i_r\leq n ,r~ \mbox{even}}{(-\lambda)^{\frac{r}{2}+1}} 
				(\mathrm{R_V}\circ f_n)(\mathrm{R(a_1),\ldots,a_{i_1},\ldots,a_{i_r},\ldots,R(a_n))}}.
		\end{align*}
	\end{defi}
	\begin{lem}
		\begin{equation}\label{equation morphism and coboboundary}
			\mathrm{\phi^{n+1}(\delta^n_\mathrm{CE}(f_n))=\delta^n_\mathrm{mRBO^\lambda}
				(\phi^n(f_n)),\quad \text{where } f_n\in C^n(A;V)}.
		\end{equation}
		
	\end{lem}
	\begin{proof}
		The proof is similar to (Lemma 4.2 of \cite{B1}) and (Proposition 5.2 of \cite{K0}).
	\end{proof}
	\begin{defi}
		Let $\mathrm{(\mathcal{A}_\mathrm{R},R)}$ be a modified Rota-Baxter Lie algebra of weight $\mathrm{\lambda}$ and 
		$\mathrm{(\widetilde{\mathcal{V}},R_\mathrm{V})}$ be a representation of it. Now, we define for $\mathrm{n\geq1}$,
		\begin{equation*}
			\mathrm{C^n_\mathrm{mRBLA^\lambda}(A;V)=
				C^n(A;V)\oplus C^{n-1}_\mathrm{mRBO^\lambda}(A;V)}
		\end{equation*}
		
		and the coboundary map is defined as
		\begin{equation*}
			\mathrm{\partial^n_{\mathrm{mRBLA^\lambda}}:C^n_\mathrm{mRBLA^\lambda}(A;V)\rightarrow 
				C^{n+1}_\mathrm{mRBLA^\lambda}(A;V)}
		\end{equation*}
		by
		\begin{equation}\label{coboundary of mRBLA}
			\mathrm{\partial^n_{\mathrm{mRBLA^\lambda}}(f_n,g_{n-1})=(\delta^n_\mathrm{CE}(f_n),-\delta^{n-1}_\mathrm{mRBO^\lambda}
				(g_{n-1})-\phi^n(f_n)),\quad \text{ for } (f_n,g_{n-1})\in C^n_\mathrm{mRBLA^\lambda}(A;V)}
		\end{equation}
	\end{defi}
	\begin{thm}
		With The above notations we have $\mathrm{\big(C^\star_\mathrm{mRBLA^\lambda}(A;V),
			\partial^\star_{\mathrm{mRBLA^{\lambda}}}\big)}$ is a cochain complex, i.e,
		\begin{equation}\label{mRBLA coboundary}
			\mathrm{\partial^{n+1}_{\mathrm{mRBLA^{\lambda}}}\circ \partial^n_{\mathrm{mRBLA^{\lambda}}}=0,\quad \forall n\geq1}.
		\end{equation}
	\end{thm}
	\begin{proof}
		For $\mathrm{(f_n,g_{n-1})\in C^n_\mathrm{mRBLA^\lambda}(A;V)}$, using the fact that 
		$\mathrm{\big(C^\star(A;V),\delta^\star_\mathrm{CE}\big)}$ and 
		$\mathrm{\big(C^\star_\mathrm{mRBO^\lambda}(A;V),\delta^\star_\mathrm{mRBO^\lambda}\big)}$ are 
		both too complex cochains and \eqref{equation morphism and coboboundary} we have
		\begin{align*}
			\mathrm{\partial^{n+1}_\mathrm{mRBLA^\lambda}\circ \partial^n_\mathrm{mRBLA^\lambda}(f_n,g_{n-1})}
			&=\mathrm{\partial^{n+1}_\mathrm{mRBLA^\lambda}\big(\delta^n_\mathrm{CE}f_n,-\delta^{n-1}_\mathrm{mRBO^\lambda}g_{n-1}-\phi^nf_n\big)}\\
			&=\mathrm{\big(\delta^{n+1}_\mathrm{CE} (\delta^n_\mathrm{CE}f_n),-\delta^{n-1}_\mathrm{mRBO^\lambda}(-\delta^{n-1}_\mathrm{mRBO^\lambda}g_{n-1}-\phi^nf_n)-\phi^n(\delta^n_\mathrm{CE}f_n)\big)}\\
			&=\mathrm{\big(0,\delta^{n-1}_\mathrm{mRBO^\lambda}(\phi^nf_n)-\phi^n(\delta^n_\mathrm{CE}f_n)\big)}\\
			&=\mathrm{0}.
		\end{align*}
	\end{proof}
	With respect to the representation $(\mathcal{V},\mathrm{R_V},\mathrm{d_V})$ we obtain a complex 
	$\mathrm{\Big(\mathfrak{C}^{\star}_{\mathrm{mRBLA^{\lambda}}}
		(A,V),\partial^\star_{\mathrm{mRBLA^{\lambda}}}\Big)}$. Let $\mathrm{Z^n_{\mathrm{mRBLA^{\lambda}}}(A,V)}$ and
	$\mathrm{B^n_{mRBLA^{\lambda}}(A,V)}$ denote the space of $\mathrm{n}$-cocycles and $\mathrm{n}$-coboundaries, respectively.
	Then we define the corresponding cohomology groups by
	\begin{equation*}
		\mathrm{\mathcal{H}^n_{\mathrm{mRBLA^{\lambda}}}(A,V):=\frac{Z^n_{\mathrm{mRBLA^{\lambda}}}(A,V)}
			{B^n_{\mathrm{mRBLA^{\lambda}}}(A,V)},\quad \text{for } n\geq1}.
	\end{equation*}
	They are called \textbf{the cohomology of modified Rota-Baxter Lie algebra} $(\mathcal{A},R)$ with coefficients in
	the representation $(\mathcal{V},R_V)$.
	
	
	\subsection{Cohomology of modified Rota-Baxter LieDer pair}
	In this subsection we introduce the cohomology of modified Rota-Baxter LieDer pair.
	\\
	Define a linear map
	
	\begin{equation*}
		\mathrm{\Delta^n:C^n_\mathrm{mRBLA^\lambda}(A;V)\rightarrow C^n_\mathrm{mRBLA^\lambda}(A;V)\quad \text{ by }}
	\end{equation*}
	\begin{equation}\label{coboundary4}
		\mathrm{\Delta^n(f_n,g_{n-1}):=(\Delta^n(f_n),\Delta^{n-1}(g_{n-1})),\quad \forall (f_n,g_{n-1})
			\in C^n_\mathrm{mRBLA^\lambda}(A;V)}.
	\end{equation}
	Where
	\begin{equation*}
		\mathrm{\Delta^n (f_n)=\displaystyle\sum_{i=1}^nf_n\circ (\mathrm{Id_A}\otimes \cdots\otimes d\otimes \cdots\otimes 
			\mathrm{Id_A})-d_{\mathrm{V}}\circ f_n}.
	\end{equation*}

	In the next proposition we show that $\mathrm{\phi^\star}$ and $\mathrm{\Delta^\star}$ are commutative, which 
	is useful in the cohomology.
	\begin{pro}
		\begin{equation}\label{coboundary3}
			\mathrm{\phi^n\circ\Delta^n=\Delta^n\circ\phi^n}.
		\end{equation}
	\end{pro}
	\begin{proof}
		For $\mathrm{f_n\in C^n(A;V)}$ and $\mathrm{a_1,\cdots,a_n\in A}$ we have
		\begin{align*}
			&\mathrm{\phi^n\circ\Delta^nf_n(\mathrm{a_1,\cdots,a_n})}\\
			=&\mathrm{\Delta^n(f_n(\mathrm{Ra_1},\cdots\mathrm{Ra_n}))}\\
			&\mathrm{-\sum _{1\leq i_1<i_2< \cdots < i_r\leq n ,r~ \mbox{odd}}{(-\lambda)^{\frac{r-1}{2}}} (R_V\circ\Delta^n f_n)(\mathrm{R(a_1),\ldots,a_{i_1},\ldots,a_{i_r},\ldots,R(a_n)})}\\
			&\mathrm{-\sum _{1\leq i_1<i_2< \cdots < i_r\leq n ,r~ \mbox{even}}{(-\lambda)^{\frac{r}{2}+1}} (R_V\circ\Delta^n f_n)(\mathrm{R(a_1),\ldots,a_{i_1},\ldots,a_{i_r},\ldots,R(a_n)})}\\
			=&\mathrm{\displaystyle\sum_{\mathrm{k=1}}^nf_n(\mathrm{Ra_1,\cdots,d\circ Ra_i,\cdots,Ra_n})-\mathrm{d_V}\circ f_n\mathrm{(Ra_1,\cdots,Ra_n)}}\\
			&\mathrm{-\sum _{1\leq i_1<i_2< \cdots < i_r\leq n ,r~ \mbox{odd}}{(-\lambda)^{\frac{r-1}{2}}}\Big( \displaystyle\sum_{\mathrm{k\neq i_1,\cdots,i_r}}(\mathrm{R_V}\circ f_n)
				\mathrm{(Ra_1,\cdots,d\circ Ra_k,\cdots,a_{i_1},\cdots,a_{i_r},\cdots,Ra_n)}}\\
			&\mathrm{+(\mathrm{R_V}\circ f_n)\displaystyle\sum_{\mathrm{p=1}}^r(\mathrm{Ra_1,\cdots,a_{i_1},\cdots,d_Va_{i_p},\cdots,a_{i_r},\cdots,Ra_n})}\\
			&\mathrm{-(\mathrm{R_V}\circ \mathrm{d_V}f_n)(\mathrm{(Ra_1,\cdots,a_{i_1},\cdots,a_{i_r},\cdots,Ra_n)})\Big)}\\
			&\mathrm{-\sum _{1\leq i_1<i_2< \cdots < i_r\leq n ,r~ \mbox{even}}{(-\lambda)^{\frac{r}{2}+1}}
				\Big( \displaystyle\sum_{\mathrm{k\neq i_1,\cdots,i_r}}(\mathrm{R_V}\circ f_n)\mathrm{(Ra_1,\cdots,
					d\circ Ra_k,\cdots,a_{i_1},\cdots,a_{i_r},\cdots,Ra_n)}}\\
			&\mathrm{+(\mathrm{R_V}\circ f_n)\displaystyle\sum_{\mathrm{p=1}}^r(\mathrm{Ra_1,\cdots,a_{i_1},
					\cdots,da_{i_p},\cdots,a_{i_r},\cdots,Ra_n})}\\
			&\mathrm{-(\mathrm{R_V}\circ \mathrm{d_V}f_n)(\mathrm{(Ra_1,\cdots,a_{i_1},\cdots,a_{i_r},\cdots,Ra_n)})\Big)}\\
			=&\mathrm{\Delta^n\circ\phi^nf_n(\mathrm{a_1,\cdots,a_n})}
		\end{align*}
	\end{proof}
	Recall that, from the cohomology of LieDer pair (\cite{R0}. Lemma(3.1)), we have
	\begin{equation}\label{coboundary1}
		\mathrm{\delta^n_\mathrm{CE}\circ\Delta^n=\Delta^{n+1}\circ\delta^n_\mathrm{CE}}.
	\end{equation}
	
	Also since $\mathrm{\mathcal{A}_{\mathrm{R}}}$ is a Lie algebra and 
	$\mathrm{\delta_\mathrm{mRBO^\lambda}^\star}$ its coboundary with respect to the representation 
	$\mathrm{\widetilde{V}=(V;\rho_R)}$ and  $\mathrm{(\mathcal{A}_{\mathrm{R}},d)}$ is a LieDer pair then we get
	\begin{equation}\label{coboundary2}
		\mathrm{\delta^n_\mathrm{mRBO^\lambda}\circ\Delta^n=	\Delta^{n+1}\circ\delta^n_\mathrm{mRBO^\lambda}}.
	\end{equation}
	\begin{pro}
		With the above notations, $\Delta^\star$ is a cochain map, i.e.
		\begin{equation}\label{coboundary5}
			\mathrm{\partial^n_\mathrm{mRBLA^\lambda}\circ \Delta^n
				=\Delta^{n+1}\circ\partial^n_\mathrm{mRBLA^\lambda}}
		\end{equation}
	\end{pro}
	\begin{proof}
		For $\mathrm{(f_n,g_{n-1})\in C^n_\mathrm{mRBLA^\lambda}(A;V)}$ and by using \eqref{coboundary1}, \eqref{coboundary2}, \eqref{coboundary3} and \eqref{coboundary4} we have
		\begin{align*}
			\mathrm{\partial^n_\mathrm{mRBLA^\lambda}(\Delta^n(f_n,g_{n-1}))}
			&=\mathrm{\partial^n_\mathrm{mRBLA^\lambda}(\Delta^nf_n,\Delta^{n-1}g_{n-1})}\\
			&=\mathrm{\Big(\delta^n_\mathrm{CE}(\Delta^nf_n),-\delta^{n-1}_\mathrm{mRBO^\lambda}
				(\Delta^{n-1}g_{n-1})-\phi^n(\Delta^nf_n)\Big)}\\
			&\mathrm{=\Big(\Delta^{n+1}(\delta^n_\mathrm
				{CE}f_n),-\Delta^n(\delta^{n-1}_\mathrm{mRBO^\lambda}g_{n-1})-\Delta^n(\phi^nf_n)\Big)}\\
			&\mathrm{=\Delta^{n+1}(\partial^n_\mathrm{mRBLA^\lambda}(f_n,g_{n-1}))}
		\end{align*}
	\end{proof}
	Using all those tools we are in position to define the cohomology of modified Rota-Baxter LieDer pair 
	$\mathrm{(\mathcal{A},R,d)}$ with coefficients in a representation $\mathrm{(\mathcal{V},R_\mathrm{V},d_\mathrm{V})}$.\\
	Denote
	\begin{equation*}
		\mathrm{\mathfrak{C}^n_{\mathrm{mRBLD^{\lambda}}}(A;V):=C^n_\mathrm{mRBLA^\lambda}(A;V)\times
			C^{n-1}_{\mathrm{mRBLA^{\lambda}}}(A;V),\quad n\geq2},
	\end{equation*}
	and
	\begin{equation*}
		\mathrm{\mathfrak{C}^1_{\mathrm{mRBLD^{\lambda}}}(A;V):=C^1(A;V)}.
	\end{equation*}
	Define a linear map
	\begin{eqnarray*}
		\mathrm{\mathfrak{D}^1_{\mathrm{mRBLD^{\lambda}}}}&:&\mathrm{\mathfrak{C}^1_{\mathrm{mRBLD^{\lambda}}}
			(A;V)\rightarrow \mathfrak{C}^2_{\mathrm{mRBLD^{\lambda}}}(A;V) \text{ by}}\\
		&&\mathrm{\mathfrak{D}^1_{\mathrm{mRBLD^{\lambda}}}(f_1)
			=(\partial^1_\mathrm{mRBLA^\lambda}(f_1),-\Delta^1(f_1)),\quad \forall f_1\in
			C^1(A;V)},
	\end{eqnarray*}
	and when $\mathrm{n\geq2}$
	
	\begin{equation*}
		\mathrm{\mathfrak{D}^n_{\mathrm{mRBLD^{\lambda}}}:\mathfrak{C}^n_{\mathrm{mRBLD^{\lambda}}}(A;V)\rightarrow
			\mathfrak{C}^{n+1}_{\mathrm{mRBLD^{\lambda}}}(A;V)}
	\end{equation*}
	is defined by
	\begin{equation}\label{mRBLD coboundary}
		\mathrm{\mathfrak{D}^n_{\mathrm{mRBLD^{\lambda}}}((f_n,g_{n-1}),(h_{n-1},s_{n-2}))
			=(\partial^n_\mathrm{mRBLA^\lambda}(f_n,g_{n-1}),\partial^{n-1}_\mathrm{mRBLA^\lambda}(h_{n-1},s_{n-2})
			+(-1)^n\Delta^n(f_n,g_{n-1}))}.
	\end{equation}
	
	\begin{thm}
		With The above notations we have 
		$\mathrm{\big(\mathfrak{C}^\star_\mathrm{mRBLD^\lambda}(A;V),\mathfrak{D}^\star_{\mathrm{mRBLD^{\lambda}}}\big)}$ is a cochain complex, i.e,
		\begin{equation*}
			\mathrm{\mathfrak{D}^{n+1}_{\mathrm{mRBLD^{\lambda}}}\circ \mathfrak{D}^n_{\mathrm{mRBLD^{\lambda}}}=0,
				\quad \forall n\geq1}.
		\end{equation*}
	\end{thm}
	\begin{proof}
		For $\mathrm{n\geq1}$, using equations \eqref{mRBLD coboundary},  \eqref{coboundary5} and \eqref{mRBLA coboundary}
		
		\begin{align*}
			&\mathrm{\mathfrak{D}^{n+1}_{\mathrm{mRBLD^{\lambda}}}\circ\mathfrak{D}^n_{\mathrm{mRBLD^{\lambda}}}
				((f_n,g_{n-1}),(h_{n-1},s_{n-2}))}\\
			=&\mathrm{\mathfrak{D}^{n+1}_{\mathrm{mRBLD^{\lambda}}}(\partial^n_\mathrm{mRBLA^\lambda}(f_n,g_{n-1}),
				\partial^{n-1}_\mathrm{mRBLA^\lambda}(h_{n-1},s_{n-2})+(-1)^n\Delta^n(f_n,g_{n-1}))}\\
			=&\mathrm{\Big(\partial^{n+1}_\mathrm{mRBLA^\lambda}(\partial^n_\mathrm{mRBLA^\lambda}(f_n,g_{n-1})),
				\partial^n_\mathrm{mRBLA^\lambda}(\partial^{n-1}_\mathrm{mRBLA^\lambda}(h_{n-1},s_{n-2})
				+(-1)^n\Delta^n(f_n,g_{n-1})}\\
			&\mathrm{+(-1)^{n+1}\Delta^{n+1}(\partial^n_\mathrm{mRBLA^\lambda}(f_n,g_{n-1})\Big)}\\
			=&\mathrm{\Big((0,0),(-1)^n\partial^n_\mathrm{mRBLA^\lambda}(\Delta^n(f_n,g_{n-1}))
				+(-1)^{n+1}\Delta^{n+1}(\partial^n_\mathrm{mRBLA^\lambda}(f_n,g_{n-1})\Big)}\\
			&=0.
		\end{align*}
		This complete the proof.
	\end{proof}
	With respect to the representation $\mathrm{(\mathcal{V},\mathrm{R_V},\mathrm{d_V})}$ we obtain a complex 
	$\mathrm{\Big(\mathfrak{C}^{\star}_{\mathrm{mRBLD^{\lambda}}}
		(A,V),\mathfrak{D}^\star_{\mathrm{mRBLD^{\lambda}}}\Big)}$. Let $\mathrm{Z^n_{\mathrm{mRBLD^{\lambda}}}(A,V)}$ and
	$\mathrm{B^n_{\mathrm{mRBLD^{\lambda}}}(A,V)}$ denote the space of $\mathrm{n}$-cocycles and $\mathrm{n}$-coboundaries, respectively.
	Then we define the corresponding cohomology groups by
	\begin{equation*}
		\mathcal{H}^n_{\mathrm{mRBLD^{\lambda}}}(\mathrm{A,V}):=\frac{Z^n_{\mathrm{mRBLD^{\lambda}}}(\mathrm{A,V})}
		{B^n_{\mathrm{mRBLD^{\lambda}}}(\mathrm{A,V})},\quad \text{for } \mathrm{n\geq1}.
	\end{equation*}
	They are called \textbf{the cohomology of modified Rota-Baxter LieDer pair} $\mathrm{(\mathcal{A},R,d)}$ with coefficients
	in the representation $\mathrm{(\mathcal{V},R_V,d_{\mathrm{V}})}$.
	
	\section{Formal deformation of a modified Rota-Baxter LieDer pair}\label{sec4}
	\def\theequation{\arabic{section}.\arabic{equation}}
	\setcounter{equation} {0}
	In this section, we study a \textbf{one-parameter formal deformation} of modified Rota-Baxter LieDer pair. We use the notation
	$\mathrm{\mu}$ for the bilinear product $\mathrm{[-,-]}$ and the adjoint representation for modified Rota-Baxter LieDer
	pair.
	\begin{defi}
		Let $\mathrm{(A,\mu,R,d)}$ be a modified Rota-Baxter LieDer pair. A one-parameter formal deformation of
		$\mathrm{(A,\mu,R,d)}$ is a triple of power series $\mathrm{(\mu_\mathrm{t},R_\mathrm{t},d_\mathrm{t})},$
		
		\begin{eqnarray*}
			\mathrm{\mu_\mathrm{t}}&=&\mathrm{\sum _{i=0}^{\infty}\mu_it^i, \quad \mu_{i} \in C^2(A,A)},\\
			\mathrm{ R_\mathrm{t}}&=& \mathrm{\sum_{i=0}^{\infty} R_it^i,\quad R_i \in C^1_{mRBO^\lambda}(A,A)},\\
			\mathrm{d_\mathrm{t}}&=&\mathrm{\sum _{i=0}^{\infty}d_it^i,\quad d_i \in C^1_{mRBLA^\lambda}(A,A)}.
		\end{eqnarray*}
		such that $\mathrm{(A[\![t]\!],\mu_\mathrm{t},R_\mathrm{t},d_\mathrm{t})}$ is a modified Rota-Baxter LieDer pair, where
		$\mathrm{(\mu_0,R_{0},d_0)=(\mu,R,d)}.$
	\end{defi}
	Therefore, $\mathrm{(\mu_\mathrm{t},R_\mathrm{t},d_\mathrm{t})}$ will be a formal one-parameter deformation of a modified
	Rota-Baxter LieDer pair
	$\mathrm{(A,\mu,R,d)}$ if and only if the following conditions are satisfied for any $\mathrm{a,b,c \in A}$
	\begin{align*}
		\mathrm{\mu_\mathrm{t}(\mu_t(a,b),c)}&\mathrm{+\mu_\mathrm{t}(\mu_\mathrm{t}(b,c),a)+\mu_\mathrm{t}(\mu_\mathrm{t}(c,a),b)=0},\\
		\mathrm{\mu_\mathrm{t}(R_\mathrm{t}(a),R_\mathrm{t}(b))}&\mathrm{-R_\mathrm{t}(\mu_\mathrm{t}(a,R_\mathrm{t}(b))-\mu_\mathrm{t}(R_\mathrm{t}(a),b))- \lambda~ \mu_\mathrm{t} (a,b)=0},\\
		\mathrm{d_\mathrm{t}(\mu_\mathrm{t}(a,b))} &\mathrm{-\mu_\mathrm{t}(d_\mathrm{t}(a),b)-\mu_\mathrm{t}(a,d_\mathrm{t}(b))=0},\\
		\mathrm{R_\mathrm{t}\circ d_\mathrm{t}}&\mathrm{-d_\mathrm{t}\circ R_\mathrm{t}=0}.
	\end{align*}
	
	Expanding the above equations and equating the coefficients of $\mathrm{t^n}$($\mathrm{n}$ non-negative integer) from
	both sides, we get
	\begin{align*}
		\mathrm{\sum _{\substack{\mathrm{i}+\mathrm{j}=n \\\mathrm{i},\mathrm{j}\geq 0}}\mu_\mathrm{i}(\mu _\mathrm{j}(a,b),c)}
		&\mathrm{+\sum _{\substack{\mathrm{i}+\mathrm{j}=n \\\mathrm{i},\mathrm{j}\geq 0}} \mu_\mathrm{i}(\mu_\mathrm{j} (b,c),a)
			+\sum _{\substack{\mathrm{i}+\mathrm{j}=n \\\mathrm{i},\mathrm{j}\geq 0}} \mu_\mathrm{i}(\mu_\mathrm{j} (c,a),b)=0},\\
		\mathrm{\sum_{\substack{\mathrm{i}+\mathrm{j}+k=n \\ \mathrm{i},\mathrm{j},k \geq 0}}\mu_\mathrm{i}(R_\mathrm{j}(a),R_k(b))}
		&\mathrm{-\sum_{\substack{\mathrm{i}+\mathrm{j}+k=n \\ \mathrm{i},\mathrm{j},k \geq 0}}R_\mathrm{i}(\mu_\mathrm{j}
			(R_k(a),b))-\sum_{\substack{\mathrm{i}+\mathrm{j}+k=n \\ \mathrm{i},\mathrm{j},k \geq 0}}R_\mathrm{i}
			(\mu_\mathrm{j}(a,R_k(b)))- \lambda \mu_n(a,b)=0},\\
		\mathrm{\sum _{\substack{\mathrm{i}+\mathrm{j}=n \\\mathrm{i},\mathrm{j}\geq 0}}d_\mathrm{i}(\mu_\mathrm{j}(a,b))}
		&\mathrm{-\sum _{\substack{\mathrm{i}+\mathrm{j}=n \\\mathrm{i},\mathrm{j}\geq 0}}\mu_\mathrm{j}(d_\mathrm{i}(a),b)
			-\mu_\mathrm{j}(a,d_\mathrm{i}(b))=0},\\
		\mathrm{\sum _{\substack{\mathrm{i}+\mathrm{j}=n \\\mathrm{i},\mathrm{j}\geq 0}}R_\mathrm{i}\circ d_\mathrm{j}}
		&\mathrm{-\sum _{\substack{\mathrm{i}+\mathrm{j}=n \\\mathrm{i},\mathrm{j}\geq 0}}d_\mathrm{i}\circ R_\mathrm{j}=0}.
	\end{align*}
	Note that for $\mathrm{n=0}$, the above equations are precisely the Jacobi identity of $\mathrm{(A,\mu)}$, the condition
	for modified Rota-Baxter operator of weight $\mathrm{\lambda}$, the condition for the derivation $\mathrm{d}$ on
	$\mathrm{(A,\mu)}$ and the condition of compatibility of $\mathrm{R}$ and $\mathrm{d}$ respectively.
	\\
	Now, putting $\mathrm{n=1}$ in the above equations, we get
	\begin{equation}\label{deformation eq1}
		\mathrm{\mu_1 (\mu(a,b),c)+\mu (\mu_1(a,b),c)+\mu_1(\mu (b,c),a) +\mu (\mu_1(b,c),a)+\mu_1 (\mu (c,a),b)
			+\mu(\mu_1 (c,a),b)=0},
	\end{equation}
	
	\begin{equation}\label{deformation eq2}
		\begin{split}
			&\mathrm{\mu_1(R(a),R(b))+\mu (R_1(a),R(b))+\mu (R(a),R_1(b))-R_1(\mu (R(a),b))-R(\mu (R_1(a),b))} \\
			&\mathrm{-R(\mu _1 (R(a),b))-R_1(\mu (a,R(b)))-R(\mu _1 (a,R(b)))-R(\mu (a,R_1(b)))-\lambda ~\mu_1 (a,b)=0},
		\end{split}
	\end{equation}
	
	\begin{equation}\label{deformation eq3}
		\mathrm{d_1(\mu(a,b))+d(\mu_1(a,b))-\mu_1(d(a),b)-\mu(d_1(a),b)-\mu_1(a,d(b))-\mu(a,d_1(b))=0},
	\end{equation}
	and
	\begin{equation}\label{deformation eq4}
		\mathrm{R_1\circ d+R\circ d_1-d_1\circ R-d\circ R_1=0}.
	\end{equation}
	Where $\mathrm{a,b,c \in A}$.\\
	From the equation \eqref{deformation eq1}, we have
	\begin{equation}\label{deformation cocycle1}
		\mathrm{\delta_\mathrm{CE}^2(\mu_1)(a,b,c)=0},
	\end{equation}
	from the equation \eqref{deformation eq2}, we have
	\begin{equation}\label{deformation cocycle2}
		\mathrm{-\delta_\mathrm{mRBO^\lambda}^1\mathrm{(R_1)(a,b)}-\phi^2\mu_1(\mathrm{a,b})=0},
	\end{equation}
	from the equation \eqref{deformation eq3}, we have
	\begin{equation}\label{deformation cocycle3}
		\mathrm{\delta_\mathrm{CE}^1\mathrm{(d_1)(a,b)}+\Delta^2\mu_1(\mathrm{a,b})=0},
	\end{equation}
	and from the equation \eqref{deformation eq4}, we have
	\begin{equation}
		\mathrm{\Delta^1\mathrm{R_1(a)}-\phi^1\mathrm{d_1(a)}=0}.
	\end{equation}
	Therefore,
	$\mathrm{\Big((\delta_\mathrm{CE}^2(\mu_1),-\delta_\mathrm{mRBO^\lambda}^1\mathrm{(R_1)}-\phi^2\mu_1),
		(\delta_\mathrm{CE}^1\mathrm{(d_1)}+\Delta^2\mu_1,\Delta^1\mathrm{R_1}-\phi^1\mathrm{d_1})\Big)=((0,0),(0,0))}$.
	Hence,\\ $\mathrm{\mathfrak{D}^2_\mathrm{mRBLD^\lambda}(\mu_1,R_1,d_1)=0}$.\\
	This proves $\mathrm{(\mu_1,\mathrm{R_1,d_1})}$ is a $\mathrm{2}$-cocycle in the cochain complex
	$\mathrm{\Big(\mathfrak{C}^\star_{mRBLD^\lambda}(A,A),\mathfrak{D}^\star_\mathrm{mRBLD^\lambda}\Big)}.$
	Thus, from the above discussion, we have the following theorem.
	\begin{thm}\label{infy-co}
		Let $\mathrm{(\mu_\mathrm{t}, \mathrm{R_t},\mathrm{d_t})}$ be a one-parameter formal deformation of a modified
		Rota-Baxter LieDer pair $\mathrm{(A,\mu,\mathrm{R,d})}$. Then $\mathrm{(\mu\mathrm{_1,R_1,d_1})}$ is a
		$\mathrm{2}$-cocycle in the cochain complex $\mathrm{\Big(\mathfrak{C}^\star_{mRBLD^\lambda}(A,A),
			\mathfrak{D}^\star_\mathrm{mRBLD^\lambda}\Big)}.$
	\end{thm}
	\begin{defi}
		The $\mathrm{2}$-cocycle $\mathrm{(\mu\mathrm{_1,R_1,d_1})}$ is called \textbf{the infinitesimal of the formal one-parameter
		deformation} $(\mathrm{\mu\mathrm{_t,R_t,d_t})}$ of the modified Rota-Baxter LieDer pair $\mathrm{(A,\mu,\mathrm{R,d})}$
		of weight $\mathrm{\lambda}$.
	\end{defi}
	\begin{defi}
		Let $\mathrm{(\mu_\mathrm{t}, \mathrm{R_t},\mathrm{d_t})}$ and $\mathrm{(\mu_\mathrm{t}^\prime,
			\mathrm{R_t}^\prime,\mathrm{d_t}^\prime)}$ be two formal one-parameter deformations of a modified Rota-Baxter LieDer pair
		$\mathrm{(A,\mu,\mathrm{R,d})}$. A formal isomorphism between these two deformations is a power series
		$\mathrm{\psi_\mathrm{t}=\sum _{\mathrm{i=0}}^{\infty}\psi_\mathrm{i} \mathrm{t^i}:
			A[\![\mathrm{t}]\!] \rightarrow A[\![\mathrm{t}]\!]}$, where $\mathrm{\psi_\mathrm{i}: A \rightarrow A}$ are linear maps
		and  $\mathrm{\psi_0=\mathrm{Id}_A}$ such that the following conditions are satisfied
		\begin{eqnarray}
			\mathrm{\psi_\mathrm{t} \circ \mu^\prime_\mathrm{t}}&=&\mathrm{\mu_\mathrm{t} \circ (\psi_\mathrm{t} \otimes \psi_\mathrm{t})},\\
			\mathrm{\psi_\mathrm{t} \circ R_\mathrm{t}^\prime}&=&\mathrm{R_{\mathrm{t}} \circ \psi_\mathrm{t}},\\
			\mathrm{\psi_\mathrm{t} \circ d_\mathrm{t}^\prime}&=&\mathrm{d_{\mathrm{t}} \circ \psi_\mathrm{t}}.
		\end{eqnarray}
	\end{defi}
	Now expanding previous three equations and equating the coefficients of $\mathrm{t^n}$ from both the sides we get
	\begin{eqnarray*}
		\mathrm{\sum_{\substack {\mathrm{i+j}=n \\ \mathrm{i,j}\geq 0}}\psi _\mathrm{i}(\mu_\mathrm{j}^\prime(\mathrm{a,b}))}&
		=&\mathrm{ \sum_{\substack {\mathrm{i+j+k}=n \\ \mathrm{i,j,k}\geq 0}}\mu_\mathrm{i}(\psi_\mathrm{j}
			\mathrm{(a),\psi_k(b)}),~~ \mathrm{a,b} \in A}.\\
		\mathrm{\sum_{\substack {\mathrm{i+j}=n \\ \mathrm{i,j}\geq 0}}\psi_\mathrm{i} \circ R^\prime_\mathrm{j}}&
		=&\mathrm{\sum_{\substack {\mathrm{i+j}=n \\ \mathrm{i,j}\geq 0}} R_\mathrm{i} \circ \psi_\mathrm{j}},\\
		\mathrm{\sum_{\substack {\mathrm{i+j}=n \\ \mathrm{i,j}\geq 0}}\psi_\mathrm{i} \circ d^\prime_\mathrm{j}}&
		=&\mathrm{ \sum_{\substack {\mathrm{i+j}=n \\ \mathrm{i,j}\geq 0}} d_\mathrm{i} \circ \psi_\mathrm{j}}.
	\end{eqnarray*}
	Now putting $\mathrm{n=1}$ in the above equation, we get\\
	\begin{eqnarray*}
		\mathrm{\mu^\prime_1(a,b)}&=&\mathrm{\mu_1(a,b)+\mu (\psi_1(a),b)+\mu (a,\psi_1 (b))-\psi_1(\mu (a,b)) ,~~ a,b \in A},\\
		\mathrm{R_1^\prime}&=&\mathrm{R_1+R \circ \psi_1-\psi _1 \circ R},\\
		\mathrm{d_1^\prime}&=&\mathrm{d_1+d \circ \psi_1-\psi _1 \circ d}.
	\end{eqnarray*}
	Therefore, we have
	\[\mathrm{(\mu_1^\prime,R_1^\prime,d_1^\prime)-(\mu_1,R_1,d_1)
		=(\delta_{\mathrm{CE}}^1(\psi_1),-\phi^1(\psi_1),-\Delta^1(\psi_1))=\mathfrak{D}^1_\mathrm{mRBLD^\lambda}(\psi_1)
		\in \mathfrak{C}^{1}_\mathrm{mRBLD^\lambda}(A,A)}.\]
	Hence, from the above discussion, we have the following theorem.
	\begin{thm}
		The infinitesimals of two equivalent one-parameter formal deformation of a modified Rota-Baxter LieDer pair
		$\mathrm{(A,\mu,R,d)}$ are in the same cohomology class.
	\end{thm}
	\begin{defi}
		A modified Rota-Baxter LieDer pair $\mathrm{(A,\mu,R,d)}$ is called \textbf{rigid} if every formal one-parameter deformation
		is trivial.
	\end{defi}
	\begin{thm}
		Let $\mathrm{(A, \mu,R,d)}$ be a modified Rota-Baxter LieDer pair Then $\mathrm{(A,\mu,R,d)}$ is rigid if
		$\mathrm{\mathcal{H}^2_\mathrm{mRBLD^\lambda}(A,A)=0}.$
	\end{thm}
	\begin{proof}
		Let $\mathrm{(\mu_t, R_t,d_t)}$ be a formal one-parameter deformation of the modified Rota-Baxter LieDer pair
		$\mathrm{(A,\mu,R,d)}$. From Theorem \ref{infy-co}, $\mathrm{(\mu_1,R_1,d_1)}$ is a $\mathrm{2}$-cocycle
		and as $\mathrm{\mathcal{H}^2_\mathrm{mRBLD^\lambda}(A,A)=0}$, thus, there exists a $1$-cochain $\psi_1\in\mathfrak{C}^1_\mathrm{mRBLD^\lambda}$ such that
		\begin{equation}\label{eqt deformation}
			(\mathrm{\mu_1,R_1,d_1})=-\mathfrak{D}^1_\mathrm{mRBLD^\lambda}(\psi_1).
		\end{equation}
		Then setting $\mathrm{\psi_t=Id + \psi_1t}$, we have a deformation $(\bar{\mu}_t,\bar{R}_t,\bar{d}_t))$, where
		\begin{eqnarray*}
			\mathrm{\bar{\mu}_t(a,b)}&=&\mathrm{\big(\psi_t^{-1} \circ \mu_t \circ (\psi_t \circ \psi_t)\big)(a,b)},\\
			\mathrm{\bar{R}_t(a)}&=&\mathrm{\big(\psi_t^{-1} \circ R_t \circ \psi_t\big)(a)},\\
			\mathrm{ \bar{d}_t(a)}&=&\mathrm{\big(\psi_t^{-1} \circ d_t \circ \psi_t\big)(a)}.
		\end{eqnarray*}
		Thus, $(\mathrm{\bar{\mu}_t,\bar{R}_t,\bar{d}_t})$ is equivalent to $(\mathrm{\mu_t,R_t,d_t})$.\\
		Moreover, we have
		\begin{eqnarray*}
			\mathrm{\bar{\mu}_t(a,b)}&=& \mathrm{Id - \psi_1t+\psi^2t^2+\cdots+(-1)^i\psi_1^it^i+\cdots) 
				(\mu_t(a+\psi_1(a)t,y+\psi_1(b)t)},  \\
			\mathrm{\bar{R}_t(a)}&=& \mathrm{Id - \psi_1t+\psi^2t^2+\cdots+(-1)^i{\psi_1}^{i}t^i+
				\cdots) (R_t(a+\psi_1(a)t))},	\\
			\mathrm{\bar{d}_t(a)}&=& \mathrm{Id - \psi_1t+\psi^2t^2+\cdots+(-1)^i{\psi_1}^{i}t^i+
				\cdots) (d_t(a+\psi_1(a)t))}.
		\end{eqnarray*}
		Then,
		\begin{eqnarray*}
			\bar{\mu}_\mathrm{t}\mathrm{(a,b)}&=&\mathrm{\mu(a,b)+\big(\mu_1(a,b)+\mu(a,\psi_1(b))+\mu(\psi_1(\psi_1(a),b)
				-\psi_1(\mathrm{\mu(a,b)}\big) +\bar{\mu}_2(a,b)t^2+\cdots},\\
			\mathrm{\bar{R}_t(a)}&=&\mathrm{R(a)+\big(R(\psi_1(a))+R_1(a)-\psi_1(R(a))\big)t+\bar{R}_2(a)t^2+\cdots,} \\
			\mathrm{\bar{d}_t(a)}&=&\mathrm{d(a)+\big(d(\psi_1(a))+d_1(a)-\psi_1(d(a))\big)t+\bar{d}_2(a)t^2+\cdots}.
		\end{eqnarray*}
		By \eqref{eqt deformation}, we have
		\begin{eqnarray*}
			\mathrm{\bar{\mu}_t(a,b)}&=& \mathrm{\mu(a,b) +\bar{\mu}_2(a,b)t^2+\cdots} ,  \\
			\mathrm{\bar{R}_t(a)}&=&\mathrm{ R(a) +\bar{R}_2(a)t^2+\cdots},\\
			\mathrm{\bar{d}_t(a)}&=& \mathrm{d +\bar{d}_2(a)t^2+\cdots} .
		\end{eqnarray*}
		
		Finaly, by repeating
		the arguments, we can
		show that $\mathrm{(\mu_t,R_t, d_t)}$ is equivalent to the trivial deformation. Hence, $\mathrm{(A,\mu,R,d)}$ is rigid.
	\end{proof}
	
	\section{Abelian extension of a modified Rota-Baxter LieDer pair}\label{sec5}
\def\theequation{\arabic{section}.\arabic{equation}}
\setcounter{equation} {0}
In this section, we study abelian extensions of modified Rota-Baxter LieDer pair and show
that they are classified by the second cohomology, as one would expect of a good cohomology theory. \\

Let $\mathrm{V}$ be any vector space. We can always define a bilinear product on $\mathrm{V}$ by $\mathrm{[u,v]_V=0}$, i.e., $\mathrm{\mu_V (u,v)=0}$ for
all $\mathrm{u,v \in V}$. If $\mathrm{R_V}$ and $\mathrm{d_V}$ be two linear maps on $\mathrm{V}$, then $\mathrm{(V,\mu_V,{R_V},d_V)}$ is a modified Rota-Baxter LieDer pairs of
weight $\mathrm{\lambda}$. Now we introduce the definition of the abelian extension of the modified Rota-Baxter LieDer pair. In the sequel we denote by $\mathrm{\mathcal{V}=(V,\mu_V)=(V,[-,-]_V)}$.
\begin{defi}
	Let $\mathrm{(A,[-,-],R,d)}$ be a modified Rota-Baxter LieDer pair and $\mathrm{V}$ be a vector space. Now a modified
	Rota-Baxter LieDer pair $\mathrm{(\hat{A},[-,-]_{\wedge},{\hat{R}},\hat{d})}$ is called an extension of
	$\mathrm{(A,[-,-],R,d)}$ by $\mathrm{(V,[-,-]_\mathrm{V},R_V,d_V)}$ if there exists a short exact sequence 
	of morphisms of modified Rota-Baxter LieDer pair
	$$\begin{CD}
		0@>>> \mathrm{(\mathcal{V},d_V)} @>\mathrm{i} >> \mathrm{(\hat{\mathcal{A}},\hat{d})} @>\mathrm{p} >> \mathrm{(\mathcal{A},d)} @>>>\mathrm{0}\\
		@. @V {\mathrm{R_V}} VV @V \hat{\mathrm{R}} VV @V \mathrm{R} VV @.\\
		0@>>> \mathrm{(\mathcal{V},d_V)} @>\mathrm{i} >> \mathrm{(\hat{\mathcal{A}},\hat{d})} @>\mathrm{p} >> \mathrm{(\mathcal{A},d)} @>>>0
	\end{CD}$$
	where $\mathrm{\mu_V (u,v)=0}$ for all $\mathrm{u,v \in V}$ and $\mathrm{\hat{\mathcal{A}}=(\hat{A},[-,-]_\wedge)}$.
\end{defi}
An extension   $\mathrm{(\hat{A},[-,-]_{\wedge},{\hat{R}},\hat{d})}$ of the modified Rota-Baxter LieDer pair
$\mathrm{(A,[-,-],R,d)}$ by $\mathrm{(V,[-,-]_\mathrm{V},R_V,d_V)}$ is called abelian if the Lie algebra $\mathcal{V}$ is abelian.\\
A section of an abelian extension $\mathrm{(\hat{A},[-,-]_{\wedge},{\hat{R}},\hat{d})}$ of the modified Rota-Baxter LieDer pair
$\mathrm{(A,[-,-],R,d)}$ by $\mathrm{(V,[-,-]_\mathrm{V},R_V,d_V)}$ consists of a linear map $\mathrm{s:A\rightarrow \hat{A}}$ such that $\mathrm{p\circ s=\mathrm{Id}}$. 
In the following, we always assume that $\mathrm{(\hat{\mathcal{A}},\hat{R},\hat{d})}$ is an abelian extension of the modified Rota-Baxter LieDer pair
$\mathrm{(A,[-,-],R,d)}$ by $\mathrm{(V,[-,-]_\mathrm{V},R_V,d_V)}$ and $\mathrm{s}$ is a section of it.
For all $\mathrm{a\in A}$, $\mathrm{u\in V}$ define a linear map $\mathrm{\rho:A\rightarrow \mathrm{gl(V)}}$ by

\begin{equation}\label{extension1}
	\mathrm{\rho(a)u:=[s(a),u]_\wedge}.
\end{equation}
\begin{pro}
	With the above notations, $\mathrm{(V,\rho,R_V,d_V)}$ is a representation of the modified Rota-Baxter LieDer 
	pair $\mathrm{(\mathcal{A},R,d)}$.
\end{pro}
\begin{proof}
	Let $\mathrm{u,v\in V}$ and $\mathrm{a\in A}$.
	\begin{align*}
		\mathrm{\rho([a,b])u}&\mathrm{=[s([a,b]),u]_\wedge}\\
		&\mathrm{=[[s(a),s(b)]_\wedge+s([a,b])-[s(a),s(b)]_\wedge,u]_\wedge}\\
		&\mathrm{=[[s(a),s(b)]_\wedge,u]_\wedge}\\
		&\mathrm{=[[s(a),u]_\wedge,s(b)]_\wedge+[s(a),[s(b),u]_\wedge]_\wedge}\\
		&\mathrm{=\rho(a)(\rho(b)u)-\rho(b)(\rho(a)u)}.
	\end{align*}
	By $\mathrm{p(s(Ra)-\hat{R}(s(a)))=Ra-R(p(s(a)))=0}$\\
	which implies
	\begin{equation*}
		\mathrm{s(Ra)-\hat{R}(s(a))\in V}.
	\end{equation*}
	Also we have
	\begin{align*}
		\mathrm{[\hat{R}(s(a)),R_V(y)]_\wedge}&\mathrm{=[\hat{R}(s(a)),\hat{R}(u)]_\wedge}\\
		&\mathrm{=\hat{R}\Big([\hat{R}(s(a)),u]_\wedge+[s(a),\hat{R}u]_\wedge\Big)+\lambda[s(a),u]_\wedge}\\
		&\mathrm{=\hat{R}\Big([\hat{R}(s(a))+s(Ra)-s(Ra),u]_\wedge+[s(a),\hat{R}u]_\wedge\Big)+\lambda[s(a),u]_\wedge}.
	\end{align*}
	And
	\begin{align*}
		\mathrm{[s(R(a)),R_V(u)]_\wedge}&\mathrm{=[\hat{R}(s(a))+s(Ra)-\hat{R}(s(a)),\hat{R}u]}\\
		&\mathrm{=[\hat{R}(s(a)),\hat{R}u]_\wedge}.
	\end{align*}
	Then we have
	\begin{align*}
		\mathrm{\rho(Ra)(R_Vu)}&\mathrm{=[s(Ra),R_Vu]_\wedge}\\
		&\mathrm{=[\hat{R}(s(a)),\hat{R}u]_\wedge}\\
		&\mathrm{=\hat{R}\Big([\hat{R}(s(a))+s(Ra)-s(Ra),u]_\wedge
			+[s(a),\hat{R}u]_\wedge\Big)+\lambda[s(a),u]_\wedge}\\
		&\mathrm{=\hat{R}\Big(\rho(Ra)u+\rho(a)\hat{R}u\Big)+\lambda\rho(a)u}\\
		&\mathrm{=R_V(\rho(Ra)u+\rho(a)R_Vu)+\lambda\rho(a)u}.
	\end{align*}
	Thanks to $\mathrm{s(da)-\hat{d}s(a)\in V}$ we have the following
	\begin{align*}
		&\mathrm{[s(da)-\hat{d}s(a),u]_\wedge=0}\\
		&\mathrm{[s(da),u]_\wedge-[\hat{d}s(a),u]_\wedge=0}\\
		&\mathrm{[s(da),u]_\wedge+[s(a),d_Vu]_\wedge-d_V([s(a),u])=0}\\
		&\mathrm{\rho(da)u+\rho(a)(d_Vu)-d_V(\rho(a)u)=0}.
	\end{align*}
	This complete the proof.
\end{proof}
For any $\mathrm{a,b\in A}$ and $\mathrm{u\in V}$, define
$\mathrm{\Theta:\wedge^2A\rightarrow V}$, $\mathrm{\chi:A\rightarrow V}$ and $\mathrm{\xi:A\rightarrow V}$ as follows
\begin{eqnarray*}
	\mathrm{\Theta(a,b)}&=&\mathrm{[s(a),s(b)]_\wedge-s([a,b])},\\
	\mathrm{\chi(a)}&=&\mathrm{\hat{d}(s(a))-s(d(a))},\\
	\mathrm{\xi(a)}&=&\mathrm{\hat{R}(s(a))-s(R(a)),\quad \forall a,b\in A}.
\end{eqnarray*}
These linear maps lead to define \\
$\mathrm{R_\xi:A\oplus V\rightarrow A\oplus V}$ and $\mathrm{d_\chi:A\oplus V\rightarrow A\oplus V}$ by
\begin{eqnarray*}
	\mathrm{R_\xi(a+u)}&=&\mathrm{R(a)+R_V(u)+\xi(a)},\\
	\mathrm{d_\chi(a)}&=&\mathrm{d(a)+d_V(u)+\chi(a)}.
\end{eqnarray*}
\begin{thm}\label{theorem ext}
	With the above notations, the quadruple $\mathrm{(A\oplus V,[-,-]_\Theta,R_\xi,d_\chi)}$ where
	\begin{equation*}
		\mathrm{[a+u,b+v]_\Theta=[a,b]+\Theta(a,b),\quad \forall a,b\in A,\quad \forall u,v\in V},
	\end{equation*}
	is a modified Rota-Baxter LieDer pair if and only if $\mathrm{(\Theta,\xi,\chi)}$ is a $\mathrm{2}$-cocycle of the modified Rota-Baxter LieDer pair $\mathrm{(\mathcal{A},R,d)}$ with coefficients in the trivial representation.
\end{thm}
\begin{proof}
	If $\mathrm{(A\oplus V,[-,-]_\Theta,R_\xi,d_\chi)}$ is a modified Rota-Baxter LieDer pair, it means that $\mathrm{(\mathcal{A},[-,-]_\Theta)}$ is a Lie algebra, $\mathrm{R_\xi}$ is a modified Rota-Baxter operator, $\mathrm{d_\chi}$ is a derivation and $\mathrm{R_\xi\circ d_\chi=d_\chi\circ R_\xi}$.\\
	The couple $\mathrm{(\mathcal{A},[-,-]_\Theta)}$ is a Lie algebra means that
	\begin{equation*}
		\mathrm{[[a+u,b+v]_\Theta,c+w]_\Theta+c.p=0}.
	\end{equation*}
	Which is exactly
	\begin{equation}
		\mathrm{\delta^2_\mathrm{CE}\Theta=0}.\label{ext1}
	\end{equation}
	And $\mathrm{R_\xi}$ is a modified Rota-Baxter operator on $\mathrm{(\mathcal{A},[-,-]_\Theta)}$ means that
	\begin{equation*}
		\mathrm{[R_\xi(a+u),R_\xi(b+v)]_\Theta=R_\xi\Big([R_\xi(a+u),b+v]_\Theta+[a+u,R_\xi(b+v)]_\Theta\Big)
			+\lambda [a+u,b+v]_\Theta}.
	\end{equation*}
	Since $\mathrm{(\mathcal{A},R)}$ is a modified Rota-Baxter Lie algebra and using equation 
	\eqref{Rep of RBDer pair3} we get
	\begin{align*}
		&\mathrm{[R_\xi(a+u),R_\xi(b+v)]_\Theta-R_\xi\Big([R_\xi(a+u),b+v]_\Theta+[a+u,R_\xi(b+v)]_\Theta\Big)-\lambda [a+u,b+v]_\Theta}\\
		=&\mathrm{[R(a)+R_V(u)+\xi(R(a)),R(b)+R_V(v)+\xi(R(b))]_\Theta-R_\xi([R(a)+R_V(u)+\xi(R(a)),b+v]_\Theta)}\\
		&-\mathrm{R_\xi([a+u,R(b)+R_V(v)+\xi(R(b))]_\Theta)-\lambda [a+u,b+v]_\Theta}\\
		=&\mathrm{[R(a),R(b)]+\Theta(R(a),R(b))-R_\xi([R(a),b]+\Theta(R(a),b))-R_\xi([a,R(b)]+\Theta(a,R(b)))-\lambda[a,b]-\lambda\Theta(a,b)}\\
		=&\mathrm{\Theta(R(a),R(b))-R_V\Theta(R(a),b)-R_V(a,R(b))-\lambda \Theta(a,b)-\xi([R(a),b]+[a,R(b)])}.
	\end{align*}
	Then $\mathrm{R_\xi}$ is a modified Rota-Baxter operator on $\mathrm{(\mathcal{A},[-,-]_\Theta)}$ if and only if
	\begin{equation*}
		\mathrm{\Theta(R(a),R(b))-R_V\Theta(R(a),b)-R_V(a,R(b))-\lambda \Theta(a,b)
			-\xi([R(a),b]+[a,R(b)])=0,\quad \forall a,b\in A}.
	\end{equation*}
	Which is exactly
	\begin{equation}\label{ext2}
		\mathrm{-\delta^1_\mathrm{mRBO^\lambda}(\xi)+\Delta^2(\Theta)=0}.
	\end{equation}
	And $\mathrm{d_\chi}$ is a derivation on the Lie algebra $\mathrm{(\mathcal{A},[-,-]_\Theta)}$ if and only if
	\begin{equation*}
		\mathrm{d_\chi[a+u,b+v]_\Theta=[d_\chi(a+u),b+v]_\Theta+[a+u,d_\chi(b+v)]_\Theta}.
	\end{equation*}
	Using the fact that $\mathrm{d}$ is a derivation on $\mathrm{\mathcal{A}}$ we get
	\begin{align*}
		&\mathrm{d_\chi([a+u,b+v]_\Theta)-[d_\chi(a+u),b+v]_\Theta-[a+u,d_\chi(b+v)]_\Theta}\\
		=&\mathrm{d_\chi([a+u,b+v]_\Theta)-[d(a)+d_V(u)+\chi(a),b+v]_\Theta-[a+u,d(b)+d_V(v)+\chi(b)]_\Theta}\\
		=&\mathrm{d_\chi([a,b]+\Theta(a,b))-[d(a),b]-\Theta(d(a),b)-[a,d(b)]-\Theta(a,d(b))}\\
		=&\mathrm{d[a,b]+d_V(\Theta(a,b))+\chi([a,b])-[d(a),b]-\Theta(d(a),b)-[a,d(b)]-\Theta(a,d(b))}\\
		=&\mathrm{d_V(\Theta(a,b))-\Theta(d(a),b)-\Theta(a,d(b))+\chi([a,b])}
	\end{align*}
	Then $\mathrm{d_\chi}$ is a derivation on the Lie algebra $\mathrm{(\mathcal{A},[-,-]_\Theta)}$ if and only if
	\begin{equation*}
		\mathrm{d_V(\Theta(a,b))-\Theta(d(a),b)-\Theta(a,d(b))+\chi([a,b])=0,\quad \forall a,b\in A}.
	\end{equation*}
	Which is exactly
	\begin{equation}
		\mathrm{\delta^1_\mathrm{CE}(\chi)+\Delta^2(\Theta)=0}.\label{ext3}
	\end{equation}
	Finally, using the equations \eqref{condition1 MRBLieDer pair} and \eqref{Rep of RBDer pair3} we get
	\begin{align*}
		&\mathrm{d_\chi\circ R_\xi(a+u)-d_\xi\circ d_\chi(a+u)}\\
		=&\mathrm{d_\chi(R(a)+R_V(u)\xi(a))-R_\xi(d(a)+d_V(u)+\chi(a))}\\
		=&\mathrm{d(R(a))+d_V(R_V(u)+\xi(a))+\chi(R(a))-R(d(a))-R_V(d_V(u)+\chi(a))-\xi(d(a))}\\
		=&\mathrm{d_V(\xi(a))-\xi(d(a))+\chi(R(a))-R_V(\chi(a))}
	\end{align*}
	then $\mathrm{R_\xi}$ and $\mathrm{d_\chi}$ commute if and only if
	\begin{equation*}
		\mathrm{d_V(\xi(a))-\xi(d(a))+\chi(R(a))-R_V(\chi(a))=0,\quad \forall a\in A}.
	\end{equation*}
	Which is exactly
	\begin{equation}\label{ext4}
		\mathrm{\Delta^1(\xi)-\phi^1(\chi)=0}.
	\end{equation}
	In conclusion, $\mathrm{(A\oplus V,[-,-]_\Theta,R_\xi,d_\chi)}$ is a modified Rota-Baxter LieDer pair if and 
	only if equations \eqref{ext1}, \eqref{ext2}, \eqref{ext3}, \eqref{ext4} hold.\\
	For the second sense, if $\mathrm{(\Theta,\xi,\chi)\in\mathfrak{C}^2_\mathrm{mRBLD^\lambda}(A;V)}$ is a 
	$\mathrm{2}$-cocycle if and only if
	\begin{equation*}
		\mathrm{\Big(\delta^2_\mathrm{CE}(\Theta),-\delta^1_\mathrm{mRBO^\lambda}(\xi)-\phi^2(\Theta),  
			\delta^1_\mathrm{CE}(\chi)+\Delta^2(\Theta),\Delta^1(\xi)-\phi^1(\chi)\Big)=0}.
	\end{equation*}
	This means that equations \eqref{ext1}, \eqref{ext2}, \eqref{ext3}, \eqref{ext4} are satisfied. \\
	Thus $\mathrm{\Big(\delta^2_\mathrm{CE}(\Theta),-\delta^1_\mathrm{mRBO^\lambda}
		(\xi)-\phi^2(\Theta),\delta^1_\mathrm{CE}(\chi)+\Delta^2(\Theta),\Delta^1(\xi)-\phi^1(\chi)\Big)=0}$ 
	if and only if $\mathrm{(A\oplus V,[-,-]_\Theta)}$ is a Lie algebra, $\mathrm{R_\xi}$ is a modified 
	Rota-Baxter operator of  $\mathrm{(A\oplus V,[-,-]_\Theta)}$, $\mathrm{d_\chi}$ is a derivation of 
	$\mathrm{(A\oplus V,[-,-]_\Theta)}$ and $\mathrm{d_\chi\circ R_\xi=R_\xi\circ d_\chi}$. This complete the proof.

\end{proof}
\begin{defi}
	Two abelian extensions $\mathrm{(\hat{A}_1,[-,-]_{\wedge_1},{\hat{R}_1},\hat{d}_1)}$ and 
	$\mathrm{(\hat{A}_2,[-,-]_{\wedge_2},{\hat{R}_2},\hat{d}_2)}$ of a modified Rota-Baxter LieDer pair 
	$\mathrm{(A,[-,-],R,d)}$ by $\mathrm{(V,[-,-]_\mathrm{V},R_V,d_V)}$ are called equivalent if there exists a 
	homomorphism of modified Rota-Baxter LieDer pairs 
	$\mathrm{\gamma: (\hat{A}_1,[-,-]_{\wedge_1},{\hat{R}_1},\hat{d}_1) \rightarrow  
		(\hat{A}_2,[-,-]_{\wedge_2},{\hat{R}_2},\hat{d}_2)}$ such that the following diagram commutes
	
	$$\begin{CD}
		0@>>> {\mathrm{(V,\mathrm{R}_V,\mathrm{d_V})}} @>\mathrm{i_1} >> \mathrm{(\hat{A}_1,\hat {\mathrm{R}}_1,\hat{d}_1)} @>\mathrm{p_1} >> \mathrm{(A,\mathrm{R},\mathrm{d})} @>>>\mathrm{0}\\
		@. @| @V \mathrm{\gamma} VV @| @.\\
		0@>>> {\mathrm{(V,\mathrm{R}_V,\mathrm{d_V})}} @>\mathrm{i_2} >> \mathrm{(\hat{A}_2,\hat {\mathrm{R}}_2,\hat{d}_2)} @>\mathrm{p_2} >> \mathrm{(A,\mathrm{R},\mathrm{d})} @>>>\mathrm{0}.
	\end{CD}$$
\end{defi}
\begin{thm}
	Abelian extensions of modified Rota-Baxter LieDer pair $\mathrm{(A,[-,-],R,d)}$ by $\mathrm{(V,[-,-]_\mathrm{V},R_V,d_V)}$
	are classified by the second cohomology $\mathrm{\mathcal{H}^2_\mathrm{mRBLD^\lambda}(A;V)}$ of the modified 
	Rota-Baxter LieDer pair $\mathrm{(A,[-,-],R,d)}$ with coefficients in the trivial representation.
\end{thm}
\begin{proof}
	Let $\mathrm{(\hat{A},[-,-]_\wedge,\hat{R},\hat{d})}$ be an abelian extension of a modified Rota-Baxter LieDer 
	pair $\mathrm{(A,[-,-],R,d)}$ by $\mathrm{(V,[-,-]_\mathrm{V},R_V,d_V)}$. Let $\mathrm{s}$ be a section of it where 
	$\mathrm{s:A\rightarrow\hat{A}}$, we already have a $\mathrm{2}$-cocycle $\mathrm{(\Theta,\xi,\chi)\in\mathfrak{C}^2_\mathrm{mRBLD^\lambda}(A;V)}$ 
	by theorem \eqref{theorem ext}.\\
	First, we prove that the cohomological class of $\mathrm{(\Theta,\xi,\chi)}$ does not depend on the choice of 
	sections. Assume that $\mathrm{(s_1,s_2)}$ are two different sections providing $\mathrm{2}$-cocycles $\mathrm{(\Theta_1,\xi_1,\chi_1)}$ 
	and $\mathrm{(\Theta_2,\xi_2,\chi_2)}$ respectively. Define a linear map $\mathrm{\mathfrak{h}:A\rightarrow V}$ by $\mathrm{\mathfrak{h}(a)=s_1(a)-s_2(a),\quad \forall a\in A}$. 
	Then
	\begin{align*}
		\mathrm{\Theta_1(a,b)}&=\mathrm{[s_1(a),s_1(b)]_\wedge-s_1([a,b])}\\
		&\mathrm{=[\mathfrak{h}(a)+s_2(a),\mathfrak{h}(a)+s_2(a)]_\wedge-\mathfrak{h}([a,b])-s_2([a,b])}\\
		&\mathrm{=[s_2(a),s_2(b)]_\wedge-\mathfrak{h}([a,b])}\\
		&\mathrm{=\Theta_2(a,b)+\delta^1_\mathrm{CE}(\mathfrak{h})(a,b)}
	\end{align*}
	And
	\begin{align*}
		\mathrm{\xi_1(a)}&\mathrm{=\hat{R}(s_1(a))-s_1(R(a))}\\
		&\mathrm{=\hat{R}(\mathfrak{h}(a)+s_2(a))-\mathfrak{h}((R(a)))-s_2(R(a))}\\
		&\mathrm{=\xi_2(a)+R_V(\mathfrak{h}(a))-\mathfrak{h}(R(a))}\\
		&\mathrm{=\xi_2(a)-\phi^1\mathfrak{h}(a)}.
	\end{align*}
	Also
	\begin{align*}
		\mathrm{\chi_1(a)}&=\mathrm{\hat{d}(s_1(a))-s_1(d(a))}\\
		&\mathrm{=\hat{d}(\mathfrak{h}(a)+s_2(a))-\mathfrak{h}((d(a)))-s_2(d(a))}\\
		&\mathrm{=\chi_2(a)+d_V(\mathfrak{h}(a))-\mathfrak{h}(d(a))}\\
		&\mathrm{=\xi_2(a)-\Delta^1\mathfrak{h}(a)}.
	\end{align*}
	Which means that
	\begin{equation*}
		\mathrm{(\Theta_1,\xi_1,\chi_1)=(\Theta_2,\xi_2,\chi_2)+\mathfrak{D}^1_\mathrm{mRBLD^\lambda}(\mathfrak{h})}.
	\end{equation*}
	So $\mathrm{(\Theta_1,\xi_1,\chi_1)}$ and $\mathrm{(\Theta_2,\xi_2,\chi_2)}$ are in the same cohomological class.\\
	Next, we show that equivalent abelian extensions give rise to the same element in $\mathrm{\mathcal{H}^2_\mathrm{mRBLD^\lambda}(A;V)}$. Let $\mathrm{(\hat{A_1},[-,-]_{\wedge_1},\hat{R_1},\hat{d_1})}$ and $\mathrm{(\hat{A_2},[-,-]_{\wedge_2},\hat{R_2},\hat{d_2})}$ 
	be two equivalent abelian extensions of a modified Rota-Baxter LieDer pair $\mathrm{(A,[-,-],R,d)}$ by 
	$\mathrm{(V,[-,-]_\mathrm{V},R_V,d_V)}$ via the homomorphism $\mathrm{\gamma}$. Assume that $\mathrm{s_1}$ is a section of 
	$\mathrm{(\hat{A_1},[-,-]_{\wedge_1},\hat{R_1},\hat{d_1})}$ and $\mathrm{(\Theta_1,\xi_1,\chi_1)}$ is the corresponding 
	$\mathrm{2}$-cocycle. Since $\mathrm{\gamma}$ being a homomorphism of modified Rota-Baxter LieDer pairs such that 
	$\mathrm{\gamma_{\vert V}=\mathrm{Id_V}}$, we have then
	\begin{align*}
		\mathrm{\xi_2(a)}&\mathrm{=\hat{R_2}(a)-s_2(R(a))}\\
		&\mathrm{=\hat{R_2}(\gamma(s_1(a)))-\gamma(s_1(R(a)))}\\
		&\mathrm{=\gamma(\hat{R_1}(s_1(a))-s_1(R(a)))}\\
		&\mathrm{=\xi_1(a)}
	\end{align*}
	and
	\begin{align*}
		\mathrm{\chi_2(a)}&\mathrm{=\hat{d_2}(a)-s_2(d(a))}\\
		&\mathrm{=\hat{d_2}(\gamma(s_1(a)))-\gamma(s_1(d(a)))}\\
		&\mathrm{=\gamma(\hat{d_1}(s_1(a))-s_1(d(a)))}\\
		&\mathrm{=\chi_1(a)}
	\end{align*}
	similarly we obtain $\mathrm{\Theta_2(a,b)=\Theta_1(a,b)}$. Thus, equivalent abelian extension give rise to the 
	same element in $\mathrm{\mathcal{H}^2_\mathrm{mRBLD^\lambda}(A;V)}$.\\
	On the other hand, given two $\mathrm{2}$-cocycles $\mathrm{(\Theta_1,\xi_1,\chi_1)}$ and 
	$\mathrm{(\Theta_2,\xi_2,\chi_2)}$, we have two abelian extensions 
	$\mathrm{(A\oplus V,[-,-]_{\Theta_1},R_{\xi_1},d_{\chi_1})}$ and 
	$\mathrm{(A\oplus V,[-,-]_{\Theta_2},R_{\xi_2},d_{\chi_2})}$ by theorem \eqref{theorem ext}. 
	Suppose that they belong to the same cohomological class in $\mathrm{\mathcal{H}^2_\mathrm{mRBLD^\lambda}(A;V)}$, 
	then the existence of a linear map $\mathrm{\mathfrak{h}:A\rightarrow V}$ such that
	\begin{equation*}
		\mathrm{(\Theta_1,\xi_1,\chi_1)=(\Theta_2,\xi_2,\chi_2)+\mathfrak{D}^1_\mathrm{mRBLD^\lambda}(\mathfrak{h})}.
	\end{equation*}
	Define $\mathrm{\gamma:A\oplus V\rightarrow A\oplus V}$ by $\mathrm{\gamma(a+u)=a+\mathfrak{h}(a)+u}$, for all $\mathrm{a\in A}$ and $\mathrm{u\in V}$.
	\begin{align*}
		\mathrm{\gamma([a+u,b+v]_{\Theta_1})-[\gamma(a+u),\gamma(b+v)]_{\Theta_2}}
		&\mathrm{=\gamma([a,b]+\Theta_1(a,b))-[a+\mathfrak{h}(a)+u,b+\mathfrak{h}(b)+v]_{\Theta_2}}\\
		&\mathrm{=\gamma([a,b]+\Theta_1(a,b))-[a,b]-\Theta_2(a,b)}\\
		&\mathrm{=[a,b]+\mathfrak{h}([a,b])+\Theta_1(a,b)-[a,b]-\Theta_2(a,b)}\\
		&\mathrm{=\Theta_1(a,b)-\Theta_2(a,b)-\delta^1_\mathrm{CE}(\mathfrak{h})(a,b)}\\
		&\mathrm{=0}
	\end{align*}
	similarly we have $\mathrm{\gamma\circ R_{\xi_1}=R_{\xi_2}\circ\gamma}$ and 
	$\mathrm{\gamma\circ d_{\chi_1}=d_{\chi_2}\circ\gamma}$. Thus $\mathrm{\gamma}$ is a 
	homomorphism of these two abelian extensions, this complete the proof.
\end{proof}

	\noindent {\bf Acknowledgment:}
	The authors would like to thank the referee for valuable comments and suggestions on this article.
	
	
\end{document}